\documentclass[12pt]{article}
\usepackage{amssymb,amsfonts,amsthm,amsmath,mathrsfs,xspace,hyperref}
\usepackage{graphicx}
\usepackage[francais,english]{babel}
\usepackage{inputenc}
\usepackage[T1]{fontenc}
\usepackage{authblk}
\usepackage[margin=1.45in, centering]{geometry}
\usepackage{csquotes}
\usepackage{color}
\usepackage{enumitem}

\usepackage{titlesec}
\titleformat{\subsection}[runin]{\normalfont\bfseries}{\thesubsection}{0.5em}{}[.]
\titleformat{\subsubsection}[runin]{\normalfont\bfseries}{\thesubsubsection}{0.5em}{}[.]

    \newcommand{\ZZ}{\ensuremath{\mathbf{Z}}}%
  \newcommand{\N}{\ensuremath{\mathbf{N}}}%
\newcommand{\QZ}{\ensuremath{\mathrm{QZ}}}%
\newcommand{\QC}{\ensuremath{\operatorname{QC}}}%

    \newcommand{\LN}{\ensuremath{\mathcal{LN}}}%
        \newcommand{\LC}{\ensuremath{\mathcal{LC}}}%

  \newcommand{\Sym}{\ensuremath{\operatorname{Sym}}}%

\newcommand{\Al}{\ensuremath{\mathbb{A}}}
\newcommand{\Full}{\ensuremath{\mathbb{F}}}
    \newcommand{\acts}{\ensuremath{\curvearrowright}}%
  \newcommand{\sub}{\ensuremath{\mathrm{Sub}}}%
  \newcommand{\homeo}{\ensuremath{\operatorname{Homeo}}}%
    \newcommand{\aut}{\ensuremath{\operatorname{Aut}}}%
    \newcommand{\rist}{\ensuremath{\operatorname{Rist}}}%
		     \newcommand{\Stab}{\ensuremath{\operatorname{Stab}}}%
\newcommand{\Res}{\ensuremath{\operatorname{Res}}}%

\newcommand{\Mon}{\mathrm{Mon}} 
\newcommand{\Comm}{\mathrm{Comm}}

\theoremstyle{definition}
  \newtheorem{defin}{Definition}[section]

\theoremstyle{plain}
  \newtheorem{thm}[defin]{Theorem}
  \newtheorem{main thm}{Theorem}
  \newtheorem{prop}[defin]{Proposition}
    \newtheorem{prop-def}[defin]{Proposition-Definition}

  \newtheorem{cor}[defin]{Corollary}
\newtheorem{lem}[defin]{Lemma}

\theoremstyle{remark}
  \newtheorem{rmq}[defin]{Remark}

  \newtheorem{example}[defin]{Example}

\title{Commensurated subgroups\\ and micro-supported actions}

\author{Pierre-Emmanuel Caprace\thanks{F.R.S.-FNRS Senior Research Associate. \texttt{pe.caprace@uclouvain.be}}}

\author{Adrien Le Boudec\thanks{CNRS Researcher. \texttt{adrien.le-boudec@ens-lyon.fr}}
	\\[0.3cm]with an appendix by Dominik Francoeur\thanks{Labex Milyon Postdoctoral Researcher. \texttt{dominik.francoeur@ens-lyon.fr}}} 
\affil{UCLouvain, 1348 Louvain-la-Neuve, Belgium}
\affil{CNRS - UMPA, ENS Lyon, France}

\date{July 12, 2021}

\begin{document}
	
\newgeometry{margin=1.1in}
	
	\maketitle

\begin{abstract}
	Let $\Gamma$ be a finitely generated group and $X$ be a minimal  compact  $\Gamma$-space. We assume that the $\Gamma$-action is micro-supported, i.e. for every non-empty open subset $U \subseteq X$, there is an element of $\Gamma$ acting non-trivially on $U$ and trivially on the complement $X \setminus U$.  We show that, under suitable assumptions, the existence of certain  commensurated subgroups
	 in $\Gamma$ yields strong restrictions on the dynamics of the $\Gamma$-action: the space $X$ has compressible open subsets, and it is an almost $\Gamma$-boundary. Those properties yield in turn restrictions on the structure of $\Gamma$:  $\Gamma$ is neither amenable nor residually finite.  Among the applications, we show that the (alternating subgroup of the) topological full group associated to a minimal and expansive Cantor action of a finitely generated amenable group has no commensurated subgroups other than the trivial ones. Similarly, every commensurated subgroup of a finitely generated branch group is commensurate to a normal subgroup; the latter assertion  relies on an appendix by Dominik Francoeur, and generalizes a result of Phillip Wesolek on finitely generated just-infinite branch groups. Other applications concern discrete groups acting on the circle, and the centralizer lattice of non-discrete totally disconnected locally compact (tdlc) groups. Our results rely, in an essential way, on recent results on the structure  of tdlc groups, on the dynamics of their micro-supported actions, and on the notion of uniformly recurrent subgroups. 
\end{abstract}

\restoregeometry


\section{Introduction}
\subsection{A simplified version of the main result}

Let $\Gamma$ be a group. Two subgroups $\Lambda, \Lambda' \leq \Gamma$ are \textbf{commensurate} if $\Lambda \cap \Lambda'$ has finite index in $\Lambda$ and $ \Lambda'$. A subgroup $\Lambda \leq \Gamma$ is \textbf{commensurated} in $\Gamma$ if all $\Gamma$-conjugates of $\Lambda$ are commensurate, or equivalently if the $\Lambda$-action on the coset space $\Gamma/\Lambda$ has finite orbits. Obviously, every normal subgroup of $\Gamma$ is commensurated, and so is every subgroup of $\Gamma$ that is commensurate to a normal subgroup (e.g. the finite subgroups, and the subgroups of finite index). In general commensurated subgroups need not be commensurate to a normal subgroup. For instance every normal subgroup of the group $\mathrm{SL}(n,\ZZ[1/p])$ is finite or finite index, but $\mathrm{SL}(n,\ZZ)$ is an infinite and infinite index commensurated subgroup. Moreover,  the problems of understanding the normal subgroups and of understanding the commensurated subgroups of a group $\Gamma$ are generally rather independent, and a complete understanding of the normal subgroups does not necessarily provide a description of the commensurated subgroups. This is for instance illustrated by the case of lattices in higher rank semisimple Lie groups: while Margulis' normal subgroup theorem says that every normal subgroup of $\Gamma$ is finite or finite index, the Margulis-Zimmer conjecture on the description of the commensurated subgroups of $\Gamma$ remains open in general (see \cite{Shalom-Willis} for partial answers). See also \S \ref{subsubsec-higman} below for another illustration.

 The study of the commensurated subgroups of a group $\Gamma$ is well-known to be closely related to the study of dense homomorphisms $\Gamma \to G$ from $\Gamma$ to a totally disconnected locally compact group (\textbf{tdlc group} hereafter). Here we call a homomorphism \textit{dense} if it has dense image in $G$. Indeed, if $\Gamma \to G$ is such a homomorphism, then the preimage in $\Gamma$ of every compact open subgroup of $G$ is a commensurated subgroup of $\Gamma$. Conversely if $\Lambda \leq \Gamma$ is a commensurated subgroup, then the Schlichting completion process provides a tdlc group $\Gamma/\! \!/\Lambda$ and a dense homomorphism $\Gamma \to \Gamma/\! \!/\Lambda$ \cite{Schlichting}. We refer to \cite[Section 3]{Shalom-Willis} for a detailed introduction to Schlichting completions (see also \cite{ReidWesolek_Forum19} for additional properties). 

Let $\Gamma$ be a group acting by homeomorphisms on a topological space $X$. The rigid stabilizer $\rist_\Gamma(U)$ of a subset $U \subseteq X$ is the pointwise fixator in $\Gamma$ of the complement of $U$ in $X$. We say that the action of $\Gamma$ on $X$ is \textbf{micro-supported} if $\rist_\Gamma(U)$ acts non-trivially on $X$ for every non-empty open subset $U$ of $X$. The term `micro-supported' was first coined in \cite{CRW-part2}, although the notion it designates has frequently appeared in earlier references, notably in the work of M.~Rubin \cite{Rubin-loc-mov-book}.

In this article we relate the existence of commensurated subgroups of a group $\Gamma$ with the topological dynamics of the micro-supported actions of $\Gamma$ on compact spaces. The following is a simplified version of our main result (see Theorem \ref{thm:Main} for a more comprehensive statement).

\begin{thm}\label{thm:Main-intro}
	Let $\Gamma$ be a finitely generated group such that every proper quotient of $\Gamma$ is virtually nilpotent. Let $X$ be a compact $\Gamma$-space such that the action of $\Gamma$  on $X$ is faithful, minimal and micro-supported. If $\Gamma$ has a commensurated subgroup  which is of infinite index and which is not virtually contained in a normal subgroup of infinite index of $\Gamma$, then the following hold:
	\begin{enumerate}[label=(\roman*)]
		\item \label{it:Main-intro-1} $X$ has a non-empty open subset which is compressible by $\Gamma$. 
		
		\item  \label{it:Main-intro2} $X$ is an almost $\Gamma$-boundary. 
	\end{enumerate}

Conclusion \ref{it:Main-intro-1} implies that $\Gamma$ is monolithic, hence not residually finite; and conclusion \ref{it:Main-intro2} implies that $\Gamma$ is not amenable.
\end{thm}

We recall some terminology. A subgroup $\Lambda \leq \Gamma$ is \textbf{virtually contained} in $\Lambda' \leq \Gamma$ if $\Lambda$ has a finite index subgroup that is contained in $\Lambda'$. A group $\Gamma$ is \textbf{monolithic} if the intersection $M$ of all non-trivial normal subgroups of $\Gamma$ is non-trivial. When this holds $M$ is called the \textbf{monolith} of $\Gamma$.

Let $X$ be a compact $\Gamma$-space. A non-empty open subset  $U$ of $X$ is \textbf{compressible} by $\Gamma$ if there exists $x \in X$ such that for every neighbourhood $V$ of $x$, there exists $\gamma \in \Gamma$ such that $\gamma(U) \subset V$. The action of $\Gamma$ on $X$ is \textbf{strongly proximal} if the orbit closure of every probability measure on $X$ contains a Dirac measure, and the action of $\Gamma$ on $X$ is a (topological) \textbf{boundary} if it is minimal and strongly proximal \cite{Furst-bnd-theory}. We will also say that the action of $\Gamma$ on $X$ is \textbf{almost strongly proximal} if $X$ admits a $\Gamma$-invariant clopen partition $X = X_1 \cup \dots \cup X_d$ such that for each $i$, the action of $\Stab_\Gamma(X_i)$ on $X_i$ is strongly proximal; and $X$ is an \textbf{almost boundary} if $X$ is minimal and almost strongly proximal. When this is the case, the clopen set $X_i$ is a $\Stab_\Gamma(X_i)$-boundary for each $i$. 


Theorem \ref{thm:Main-intro}, as well as some intermediate results that we establish towards its proof, have several types of applications. The rest of this introduction is aimed at describing them.

\subsection{Topological full groups} If $\Lambda$ is a group acting on a compact space $X$, the topological full group $\Full(\Lambda,X)$ is the group of homeomorphisms $g$ of $X$ such that for every $x \in X$ there exist a neighbourhood $U$ of $x$ and an element $\gamma \in \Lambda$ such that $g(y) = \gamma(y)$ for every $y \in U$. Topological full groups were first studied in detail by Giordano--Putnam--Skau \cite{GPS99} and Matui \cite{Matui06}, and more recently in \cite{Grig-Med,Jus-Mon,Jus-N-dlS,MB-urs-full} (see also the survey \cite{Cor-Bourbaki} for additional references and historical remarks).

We denote by $\Al(\Lambda,X) \leq \Full(\Lambda,X)$ the \textbf{alternating full group}  introduced and studied by Nekrashevych in \cite{Nek-simple-dyn, Nek-pal}. If $ \Lambda$ is a finitely generated group and $ \Lambda \acts X$ is a minimal and expansive action on a Cantor space $X$, Nekrashevych showed that the group $\Al(\Lambda,X)$ is the monolith of $\Full(\Lambda,X)$, and is a finitely generated and simple group (see Section~\ref{subsec-full-gp} below for more details). When $\Lambda = \ZZ$ the alternating full group coincides with the derived subgroup of $\Full(\Lambda,X)$, and in that case finite generation and simplicity of $\Full(\Lambda,X)'$ were previously obtained by Matui in \cite{Matui06}. We refer to Section \ref{subsec-full-gp} for   definitions and details. 


Applying Theorem \ref{thm:Main-intro}, we obtain the following result.

\begin{thm} \label{thm-fullgroup-intro}
	Let $\Lambda$ be a finitely generated group, and $ \Lambda \acts X$ a minimal and expansive action on a Cantor space $X$ such that $\Al(\Lambda,X) \acts X$ does not admit any compressible open subset. Then every proper commensurated subgroup of the alternating full group $\Al(\Lambda,X)$ is finite.
\end{thm}

We recall that there exist examples of alternating full groups $\Al(\Lambda,X)$ such that the action of $\Al(\Lambda,X)$ on $X$ admits compressible open subsets, and such that $\Al(\Lambda,X)$ admits commensurated subgroups that are infinite and of infinite index. This is for instance the case of the Higman-Thompson groups $V_{d,k}$ from \cite{Hig-fp}, which admit a commensurated subgroup such that the associated Schlichting completion is a Neretin group of almost automorphisms of a tree \cite{Cap-deMedts}. Moreover, it was proven in \cite{Lederle-complet} that $V_{d,k}$ admits infinitely many commensurated subgroups. Although Theorem \ref{thm-fullgroup-intro} does not say anything about such situations,  we prove in Theorem~\ref{thm-full-tdlc-extends} that for every Schlichting completion $G$ of $\Al(\Lambda,X)$, the action of $\Al(\Lambda,X)$ on $X$ extends to a continuous  action of $G$ on $X$. 

An important situation where Theorem \ref{thm-fullgroup-intro} applies is when $X$ admits a probability measure that is invariant under the action of $\Lambda$. Such a probability measure is also invariant under the action of the full group $\Full(\Lambda,X)$. Since the existence of an invariant probability measure prevents the existence of a compressible open subset, we have the following:

\begin{cor} \label{cor-fullgroup-intro}
	Let $\Lambda$ be a finitely generated group, and $ \Lambda \acts X$ be a minimal and expansive action on a Cantor space $X$. If $X$ carries a $\Lambda$-invariant probability measure (e.g.\ if $\Lambda$ is amenable), then every proper commensurated subgroup of the alternating full group $\Al(\Lambda,X)$ is finite. 
\end{cor}

Corollary \ref{cor-fullgroup-intro} implies that for every minimal and expansive action of $\mathbb{Z}^d$ on a Cantor space, every proper commensurated subgroup of $\Al(\mathbb{Z}^d,X)$ is finite. Recall that when $d=1$, Juschenko and Monod showed that the group $\Full(\mathbb{Z},X)$ is amenable \cite{Jus-Mon}. One motivation for studying specifically the commensurated subgroups of finitely generated infinite simple amenable groups comes from the fact that, if such a group $\Gamma$ were known to admit an infinite proper commensurated subgroup $\Lambda$, then by results of \cite{CaMo-decomp}, the Schlichting completion $\Gamma/\! \!/\Lambda$ would admit as a quotient a compactly generated tdlc group that is non-discrete, topologically simple and amenable (see Proposition 3.6 in \cite{CRW-part2}). As of now, no such example is available (see Question 3 in \cite{CRW-part2}). Hence Corollary~\ref{cor-fullgroup-intro} implies that the above strategy to build such a group cannot work by starting with groups such as $\Al(\Lambda,X)$. Corollary~\ref{cor-fullgroup-intro} also answers a question raised at the end of \cite{Cap-ECM}.	

\begin{rmq}
The assumptions that $\Lambda$ is finitely generated and that the action of $\Lambda$ on $X$ is expansive are both essential; see Example \ref{ex-odom}.
\end{rmq}

\begin{rmq} 
The conclusions of Theorem \ref{thm-fullgroup-intro} and Corollary \ref{cor-fullgroup-intro} hold more generally for an arbitrary subgroup of $\Full(\Lambda,X)$ that contains $\Al(\Lambda,X)$, with appropriate modifications in the conclusion. See Theorem \ref{thm-fullgroup-no-commens}.
\end{rmq}


\subsection{Branch groups}

As mentionned above, every subgroup of a group $\Gamma$ that is commensurate to a normal subgroup is commensurated. A group $\Gamma$ in which every commensurated subgroup is commensurate to a normal subgroup may thus be viewed as a group with as few   commensurated subgroups as possible. That property was called \enquote{inner commensurator-normalizer property} in \cite{Shalom-Willis}. It is equivalent to the fact that every Schlichting completion of $\Gamma$ is compact-by-discrete (see Lemma \ref{lem:Schlichting}). Recall that a group $G$ is compact-by-discrete if $G$ admits a compact normal subgroup $K$ such that $G/K$ is discrete.

We refer the reader to Section \ref{subsec-branch} for basic definitions and to \cite{Grig-ji-branch, branch} for a general introduction to branch groups.


\begin{thm} \label{thm-branch-comm-intro}
Let $\Gamma$ be a finitely generated branch group. Then every commensurated subgroup of $\Gamma$ is commensurate to a normal subgroup of $\Gamma$.
\end{thm}

Theorem \ref{thm-branch-comm-intro} recovers and extends a result of Wesolek \cite{Wes-branch}, who showed the same result under the additionnal assumption that $\Gamma$ is just-infinite (i.e.\ every non-trivial normal subgroup of $\Gamma$ is of finite index). The proof of Theorem~\ref{thm-branch-comm-intro} relies on Theorem~\ref{thm:Main-intro}. It also uses the fact that every normal subgroup of a finitely generated branch group is finitely generated: that property is established by D.~Francoeur in the appendix.

\subsection{Groups acting on the circle}

We refer the reader to \cite{Ghys-circ} for an introduction to group actions on the circle. The following result does not formally follow from our main result. Its proof rather relies  on intermediate results that we obtain in the course of the proof of Theorem \ref{thm:Main}, suitably combined  with the Ghys--Margulis theorem about the topological dynamics of group actions on the circle (see Theorem \ref{thm-Margulis}).

\begin{thm} \label{thm-circle-intro}
	Let $\Gamma$ be a finitely generated group with a faithful, minimal, micro-supported action on the circle, such that the subgroup of $\Gamma$ generated by the elements that fix pointwise an open interval has finite index in $\Gamma$. Then every commensurated subgroup of $\Gamma$ is commensurate to a normal subgroup of $\Gamma$.
\end{thm}

See Example \ref{ex-T-circle} for a family of groups to which this result applies. In particular, we recover   the fact, first established in  \cite{LB-Wes},  that every proper commensurated subgroup of Thompson's group $T$ is finite.

\subsection{Micro-supported actions of locally compact groups}

Our results also have applications to the study of non-discrete locally compact groups. We denote by $\mathscr  S_{\mathrm{td}}$ the class of compactly generated tdlc groups that are non-discrete and topologically simple. We refer to \cite[Appendix A]{CRW-part2} for a description of various families of examples of groups within this class. Recall from \cite{CRW-part1} that every group $G$ in $\mathscr  S_{\mathrm{td}}$ (and more generally every tdlc group satisfying a condition called [A]-semisimplicity, see Section \ref{subsec-LCG}) has an associated $G$-Boolean algebra $\LC(G)$ called the \textbf{centralizer lattice} of $G$, with the property that the action of $G$ on the Stone space $\Omega_G$ of $\LC(G)$ is minimal and micro-supported, and such that every minimal and micro-supported totally disconnected compact $G$-space is a factor of $\Omega_G$ \cite[Theorem~5.18]{CRW-part1}. In Section  \ref{subsec-dense-LC} we obtain a strengthening of this universal property, by showing that  $\Omega_G$ is the maximal highly proximal extension in the sense of Auslander and Glasner \cite{Ausl-Glasn} of any  minimal and micro-supported compact $G$-space (Theorem \ref{thm:Omega_G}). It follows that every (not necessarily totally disconnected) faithful micro-supported compact $G$-space is a factor of $\Omega_G$, and is thus minimal and strongly proximal (see Corollary~\ref{cor:Omega_G}). Those results hold for all groups in the class $\mathscr R$ of \textit{robustly monolithic groups} introduced in \cite{CRW_dense} (the definition is recalled in Section~\ref{subsec-dense-LC} below). That class strictly  contains $\mathscr  S_{\mathrm{td}}$.

In Section  \ref{subsec-dense-LC} we also consider the situation of a pair of groups $(H,G)$ such that there exists a continuous homomorphism $\varphi\colon H \to G$ with dense image.  Examples of such pairs $(H,G)$ naturally arise in the study of the local structure of groups in the class $\mathscr  S_{\mathrm{td}}$, see \cite[Theorem~1.2]{Reid_Sylow} and \cite{CRW_dense}. In this situation, we study the relationship between the centralizer lattice of $H$ and that of $G$. The following is a simplified version of Theorem~\ref{thm:Dense-embedding-simple}.

\begin{thm}\label{thm:tdlc-intro}
	Let $G, H \in \mathscr  S_{\mathrm{td}}$ and  $\varphi \colon H \to G$ be a continuous injective homomorphism with dense image. If $\Omega_H$ is non-trivial, then the $H$-action on $\Omega_G$ is micro-supported,  and there exists a $H$-map $\Omega_H \to \Omega_G$ that is a highly proximal extension. In particular $\Omega_G$ is non-trivial.
\end{thm}

This result applies for instance when $H$ is the commensurator in $G$ of an infinite pro-$p$ Sylow subgroup of a compact open subgroup of $G$ (Corollary \ref{cor:Localizations}). We refer to Section  \ref{subsec-dense-LC} for details.

\subsection*{Acknowledgements}
We are grateful to Yves de Cornulier for his comments on an earlier version of this paper. This work was micro-supported by a PEPS grant from CNRS and by the project ANR-19-CE40-0008-01 AODynG. We also thank the organizers of the following conferences, and the support of the associated host institutions: \textit{Measurable, Borel, and Topological Dynamics} in October 2019 at CIRM, and \textit{Groups, Dynamics, and Approximation} in December 2019 at the Mathematisches Forschungsinstitut Oberwolfach. 

The work of Dominik Francoeur was performed within the framework of the Labex Milyon (ANR-10- LABX-0070) of Universite de Lyon, within the program "Investissements d'Avenir" (ANR-11-IDEX-0007) operated by the French National Research Agency (ANR).

\section{Preliminaries} \label{sec-prelim}

\subsection{Basic notions}
Let $G$ be a group acting by homeomorphisms on a Hausdorff space $X$. Recall that the action of $G$ on $X$ is micro-supported if the rigid stabilizer $\rist_G(U)$ acts non-trivially on $X$ for every non-empty open subset $U \subseteq X$ (by extension when $G$ is clear from the context we will also say that the $G$-space $X$ is micro-supported). Note that for faithful actions this is equivalent to saying that $\rist_G(U)$ is non-trivial. Note also that if the $G$-action  on $X$ is micro-supported, then $X$ cannot have isolated points. 

The following lemma is immediate from the definitions.

\begin{lem} \label{lem-MS-factor-faith}
	Let $G$ be a group, $X,Y$ topological spaces on which $G$ acts by homeomorphisms, and $\pi \colon Y \to X$ a continuous surjective $G$-equivariant map. Suppose that the action on $Y$ is micro-supported and the action on $X$ is faithful. Then the action on $X$ is micro-supported.
\end{lem}

\begin{proof}
	If $U$ is a non-empty open subset of $X$, then $V = \pi^{-1}(U)$ is also open and non-empty, so it follows that $\rist_G(V)$ is non-trivial since $Y$ is micro-supported. Since $\rist_G(V) \leq \rist_G(U)$, it follows that $\rist_G(U)$ is non-trivial, and hence $\rist_G(U)$ acts non-trivially on $X$ since the action on $X$ is faithful.
\end{proof}


If $G$ is a topological group and $X,Y$ are compact $G$-spaces, a continuous $G$-equivariant map $\pi: Y \to X$ will be called a \textbf{$G$-map}. When $\pi$ is onto, we say that $Y$ is an \textbf{extension} of $X$, and that $X$ is a \textbf{factor} of $Y$. We will also say that $\pi: Y \to X$ is an extension or a factor map. A $G$-map $\pi: Y \to X$ that is also a homeomorphism will be called an isomorphism. 

A non-empty subset $C \subset X$ is \textbf{compressible} if there exists $x \in X$ such that for every neighbourhood $U$ of $x$, there exists $g \in G$ such that $g(C) \subset U$. Note that by definition a compressible subset is necessarily non-empty.

\subsection{Preliminaries on highly proximal extensions}

In the sequel we will make use of the notion of highly proximal extension \cite{Ausl-Glasn}. Let $G$ be a topological group, and $X,Y$ compact $G$-spaces. An extension  $\pi: Y \to X$ is \textbf{highly proximal} if for every non-empty open subset $U \subseteq Y$, there exists $x \in X$ such that $\pi^{-1}(x) \subseteq U$. When $X,Y$ are minimal, this is equivalent to saying that for some (every) $x \in X$,  $\pi^{-1}(x)$ is compressible in $Y$ \cite{Ausl-Glasn}. 


Auslander and Glasner showed that every minimal compact $G$-space $X$ admits a highly proximal extension $\pi: X^* \to X$ with the property that for any highly proximal extension $p: Y \to X$ there exists a $G$-map $p': X^*  \to Y$ such that $\pi = p \circ p'$, and such that any highly proximal extension of $X^*$ is an isomorphism. Such a $G$-space $X^*$ is necessarily unique up to isomorphism, and is called the \textbf{maximal highly proximal extension of $X$} \cite{Ausl-Glasn}.

Following  \cite{Ausl-Glasn}, we say that $X,Y$ are \textbf{highly proximally equivalent}, denoted $X \sim_{hp} Y$, if $X$ and $Y$ admit a common highly proximal extension. Since any highly proximal extension $Z \to X$ gives rise to an isomorphism between $Z^*$ and $X^*$, being highly proximally equivalent is indeed an equivalence relation. We refer to \cite{Ausl-Glasn} for details.

\begin{lem} \label{lem-highly-prox-preserve}
Let $X,Y$ be minimal compact $G$-spaces that are highly proximally equivalent. Then $X$ is an almost $G$-boundary if and only if $Y$ is an almost $G$-boundary.
\end{lem}

\begin{proof}
By considering a common highly proximal extension, we see that it is enough to prove the statement for every highly proximal extension $\pi: Y \to X$.

Suppose that the $G$-action on $X$ is almost strongly proximal, and let $X = X_1 \cup \dots \cup X_d$ be a $G$-invariant clopen partition such that for each $i$, the action of $G_i := \Stab_G(X_i)$ on $X_i$ is a boundary action. Let $Y_i = \pi^{-1}(X_i)$. Then $Y = Y_1 \cup \dots \cup Y_d$ is a $G$-invariant clopen partition of $Y$, and the restriction $\pi|_{Y_i} \colon Y_i \to X_i$ is a highly proximal extension between the $G_i$-spaces $Y_i$ and $X_i$. Therefore the minimality of $X_i$ under the action of $G_i$ is inherited by $Y_i$, and so is strong proximality by \cite[Lemma 5.2]{Glasner-compress} (the proof is given there for almost one-to-one extensions, but the same argument applies to highly proximal extensions).

Conversely, since minimality and strong proximality are inherited by factors, it is clear that if $Y$ admits a clopen partition as above, then the push-forward of this partition in $X$ via $\pi$ provides a partition of $X$ (although distinct blocks might have the same image) with the required properties.
\end{proof}

The following shows that the properties of being micro-supported and of having compressible open subsets behave nicely with respect to highly proximal extensions.

\begin{prop} \label{prop-invariants-hpeq}
	Let $\pi \colon Y \to X$ be a highly proximal extension between compact $G$-spaces. Then the following hold:
	\begin{enumerate}[label=(\roman*)]
		\item \label{item-rist-hp-ext} If $A \subseteq X$, then $\rist_G(A) \leq \rist_G(\pi^{-1}(\overline{A}))$.
		\item \label{item-MS-hp-eq} $X$ is micro-supported if and only if $Y$ is micro-supported.
		
	\end{enumerate} 
If in addition $X,Y$ are minimal, then:
	\begin{enumerate}[resume*]
		\item \label{item-compress-hp-ext} If $A \subseteq X$ is compressible, then $\pi^{-1}(A)$ is compressible.
		\item \label{item-compress-hp-open} $X$ admits a compressible open subset if and only if $Y$ does.
	\end{enumerate}
\end{prop}

\begin{proof}
	\ref{item-rist-hp-ext}	
	Let $g \in \rist_G(A)$. Suppose for a contradiction that there exists $y \in Y$ such that $y \notin \pi^{-1}(\overline{A})$ and $g(y) \neq y$. Since $\pi^{-1}(\overline{A})$ is closed, we may find an open neighbourhood $W$ of $y$ such that $W \cap g(W) = \varnothing$ and $W \cap \pi^{-1}(\overline{A})  = \varnothing$. Since the extension is highly proximal, there exists $x \in X$ such that $ \pi^{-1}(x) \subseteq W$, and $x \notin A$ because $x$ belongs to $\pi(W)$, which is disjoint from $A$. Moreover we have $g(x) \neq x$ because $ g(\pi^{-1}(x)) \cap \pi^{-1}(x)$ is empty, so we deduce that $g \notin \rist_G(A)$, which is our contradiction.
	
	\ref{item-MS-hp-eq} 
	Suppose $X$ is micro-supported. Let $U$ be a non-empty open subset of $Y$. Denote by $V$ the set of those points $x \in X$ such that $\pi^{-1}(x) \subseteq U$. The subset $V$ is open, and $V$ is non-empty since the extension is  highly proximal. Let now $W$ be a non-empty open subset of $V$ such that $\overline{W} \subset V$. We have $\pi^{-1}(\overline{W}) \subseteq \pi^{-1}(V) \subseteq U$. According to \ref{item-rist-hp-ext}, one has $\rist_G(W) \leq \rist_G(U)$. Now since $X$ is micro-supported, one can find $x \in W$ and $g \in \rist_G(W)$ such that $g(x) \neq x$,  and it follows that $g(y) \neq y$ for every $y \in \pi^{-1}(x)$. So we have shown that  $\rist_G(U)$ acts non-trivially on $U$, and since $U$ was arbitrary, it follows that $Y$ is micro-supported. 
	
	For the converse implication, we let $V$ be a non-empty open subset of $X$, and write $U = \pi^{-1}(V)$. If $Y$ is micro-supported then one can find $g \in \rist_G(U)$ and $W$ an open subset of $U$ such that $W$ and  $g(W)$ are disjoint. Then any point $x \in V$ such that $\pi^{-1}(x) \subseteq W$ is moved by $g$, and since $g \in \rist_G(V)$, it follows that $\rist_G(V)$ acts non-trivially on $V$, and $X$ is micro-supported.
	
	\ref{item-compress-hp-ext} Suppose that $A \subseteq X$ is compressible, and write $B = \pi^{-1}(A)$.  Let $U$ be a non-empty open subset of $Y$, and let $V$ the set of points $x \in X$ such that $\pi^{-1}(x) \subseteq U$. Again $V$ is open and non-empty, so using that $A$ is compressible in $X$ and minimality, we find $g \in G$ such that $g(A) \subseteq V$. It follows that $g(B) \subseteq \pi^{-1}(V) \subseteq U$. Since $U$ was arbitrary, $B$ is compressible.
	
	\ref{item-compress-hp-open} If $U$ is a compressible open subset of $X$, then by $ V = \pi^{-1}(U)$ is a compressible open subset of $Y$ by \ref{item-compress-hp-ext}. Conversely if $V$ is compressible and open in $Y$, then $\pi(V)$ is a compressible subset of $X$. Now the image of an open subset by a factor map between minimal compact $G$-spaces always has non-empty interior, since the target space is covered by finitely many translates of that image. So there exists a non-empty open subset $W$ of $X$ in $\pi'(V)$, and it follows that $W$ is a compressible open subset of $X$.	
\end{proof}

\begin{lem} \label{lem-ext-MS-necessarilyHP}
	Let $\pi \colon Y \to X$ be an extension between faithful micro-supported compact $G$-spaces. Then $\pi$ is highly proximal.
\end{lem}

\begin{proof}
Let $U$ be a non-empty open subset of $Y$. The subgroup $\rist_G(U)$ is non-trivial, and hence acts non trivially on $X$ since the $G$-action on $X$ is faithful. If $x$ is a point of $X$ that is moved by some element $g$ of $\rist_G(U)$, then $g$ sends $\pi^{-1}(x)$ disjoint to itself. Since $g$ is supported in $U$, this implies $\pi^{-1}(x) \subseteq U$. Since $U$ was arbitrary, this proves that $\pi$ is highly proximal.
\end{proof}

When $X$ is a compact $G$-space, we denote by $2^X$ the set of all closed subsets of $X$, endowed with the Hausdorff topology. The space $2^X$ is also a compact $G$-space. Recall that the \textbf{universal minimal flow} of $G$ is the unique minimal compact $G$-space $M$ with the property that every minimal compact $G$-space is a factor of $M$. We are grateful to Todor Tsankov for pointing out to us the reference \cite[Prop.\ 3.4]{Pestov-TAMS}.

\begin{prop} \label{prop-HPM-totdisc}
Let $G$ be a tdlc group, and $X$ a minimal compact $G$-space. Then the maximal highly proximal extension $X^*$ of $X$ is a totally disconnected space.
\end{prop}

\begin{proof}
According to \cite[Theorem I.1]{Ausl-Glasn}, the $G$-space $X^*$ is isomorphic to a subsystem of $2^M$, where $M$ is the universal minimal flow of $G$. Now since the group $G$ is totally disconnected, the compact open subgroups of $G$ form a basis of neighbourhoods of the identity by van Dantzig's theorem, and hence it follows from \cite[Prop.\ 3.4]{Pestov-TAMS} that the compact $G$-space $M$ is totally disconnected. Therefore  $2^M$ is also totally disconnected, and consequently $X^*$ as well by the above property. 
\end{proof}

\subsection{Rubin's theorem}

A given group $\Gamma$ may very well admit minimal and micro-supported actions on compact spaces $X$ and $Y$ that are not isomorphic. For example Thompson's group $T$ admits such actions respectively on the circle and on a Cantor set \cite{CFP}. More generally if $\Gamma$ admits a minimal and micro-supported action on a Cantor set $X$, then by Proposition \ref{prop-invariants-hpeq} the action of $\Gamma$ on the maximal highly proximal extension $X^*$ remains micro-supported, and $X$ and $X^*$ are never isomorphic, because $X$ and $X^*$ are not even homeomorphic (see the discussion following Theorem~\ref{thm:Rubin} below). However we have the following theorem of Rubin \cite{Rubin-loc-mov-book}. If $X$ is a topological space, we denote by $R(X)$ the Boolean algebra of regular open subsets of $X$.  

\begin{thm}\label{thm:Rubin}
	Suppose that a group $\Gamma$ admits a faithful and micro-supported action on topological spaces $X,Y$. Then there exists a $\Gamma$-equivariant isomorphism $R(X) \to R(Y)$.
\end{thm}

For a compact space $X$, we denote by $\tilde{X} = \mathfrak{S}(R(X))$ the Stone space of $R(X)$. Since $R(X)$ is a complete Boolean algebra, $\tilde{X}$ is a compact extremally disconnected space (the closure of every open subset is open). There is a natural map $\pi_X: \tilde{X} \to X$, which associates to every ultrafilter on $R(X)$ its limit in the space $X$. This map is continuous and surjective, and has the property that every non-empty open subset of $\tilde{X}$ contains a fiber $\pi^{-1}(x)$ for some $x \in X$ \cite{Gleason-proj}. Observe that every homeomorphism $h$ of $X$ induces a homeomorphism of $\tilde{X}$, that we still denote $h$, such that $\pi_X \circ h = h \circ \pi_X$. 

If $\Gamma$ is a discrete group and $X$ a minimal compact $\Gamma$-space, the above discussion says that $\tilde{X}$ is also a minimal compact $\Gamma$-space, and $\pi_X: \tilde{X} \to X$ is a highly proximal extension. Moreover since $\tilde{X}$ is extremally disconnected, it is actually the maximal highly proximal extension of $X$ (\cite[Lemma 2.3]{Gleason-proj}). (See \cite{Zucker-genpt} for a generalization to non-discrete groups). The following is a consequence of Rubin's theorem and Proposition \ref{prop-invariants-hpeq}.

\begin{thm} \label{thm-MS-all-eq}
Let $\Gamma$ be a discrete group, and $X$ a compact $\Gamma$-space  that is faithful and micro-supported. Then: 
\begin{enumerate}[label=(\roman*)]
	\item \label{item-Rubin1} The action of $\Gamma$ on $\tilde{X}$ is faithful and micro-supported.
	\item \label{item-Rubin2} For every compact $\Gamma$-space $Y$ that is faithful and micro-supported, there is a factor map $\tilde{X} \to Y$ that is highly proximal.
	\item \label{item-Rubin3} If $X$ is minimal then $\tilde{X}$ is minimal. Hence all compact $\Gamma$-spaces that are faithful and micro-supported are minimal and highly proximally equivalent.
	\item \label{item-Rubin4} If $X$ is minimal and admits a compressible open subset, then all  compact $\Gamma$-spaces that are faithful and micro-supported admit a compressible open subset.
\end{enumerate}
\end{thm}

\begin{proof}
The $\Gamma$-action on $\tilde{X}$ is faithful because the $\Gamma$-action on $X$ is faithful. Since $\pi_X: \tilde{X} \to X$ is highly proximal and $X$ is micro-supported, the action of $\Gamma$ on $\tilde{X}$ is micro-supported by Proposition \ref{prop-invariants-hpeq}. This shows \ref{item-Rubin1}. 

Let $Y$ be a compact $\Gamma$-space that is faithful and micro-supported, and let $\pi_Y: \tilde{Y} \to Y$. If $\varphi: \tilde{X} \to \tilde{Y}$ is an isomorphism that is provided by the conclusion of Rubin's theorem, then $\pi_Y \circ \varphi$ is factor map $\tilde{X} \to Y$ that is highly proximal. So statement \ref{item-Rubin2} holds, and \ref{item-Rubin3} and \ref{item-Rubin4} follow from Proposition \ref{prop-invariants-hpeq}.
\end{proof}

\section{Structure theory of tdlc groups}

\subsection{The $\overline{\text{FC}}$-radical}

\begin{defin}
For a locally compact group $G$, we denote by $B(G)$ the characteristic subgroup of $G$ consisting of the elements with a relatively compact conjugacy class. 
\end{defin}

We will invoke the following result \cite[Theorem 8]{Usakov63}.

\begin{thm} \label{thm-usakov}
	If $G$ is compactly generated tdlc group and $G=B(G)$, then $G$ has a compact open normal subgroup $K$ such that $G/K$ is abelian and torsion-free. 
\end{thm}

We will also need the following, see \cite{Moller-BG}.

\begin{thm} \label{thm-b(g)-closed}
	If $G$ is compactly generated tdlc group, then $B(G)$ is a closed subgroup of $G$.
\end{thm}

\subsection{The quasi-center}\label{sec:QZ}

Given a tdlc group $G$, we denote by $\QZ(G)$ the set of elements whose centralizer is open. This is a (possibly non-closed) topologically characteristic subgroup of $G$ called the \textbf{quasi-center}. It contains all discrete normal subgroups of $G$. Moreover, if $H$ is an open subgroup of $G$, we have $\QZ(H)\leq \QZ(G)$.

Given a closed subgroup $H$ of a tdlc group $G$, we denote by $C_G(H)$ the centralizer of $H$ in $G$. We also denote by $\QC_G(H)$ the subgroup of $G$ consisting of those $g \in G$ such that $C_H(g)$ is relatively open in $H$. In particular, if $H$ is open we have $\QC_H(H) = \QZ(H)$ and $\QC_G(H) = \QZ(G)$.

\subsection{The upper structure of compactly generated tdlc groups}

We say that a locally compact group $G$ is \textbf{just-non-compact} if $G$ is not compact and every proper quotient of $G$ by a closed normal subgroup is compact. Just-non-compact groups arise naturally as quotient groups of compactly generated non-compact tdlc groups. 

\begin{prop}\label{prop:JNC-quotients}
	Let $G$ be a non-compact, compactly generated tdlc group. Then $G$ has a closed normal subgroup $N$ such that $G/N$ is just-non-compact. If in addition $G$ does not have any non-trivial finite quotient, then the quotient $G/N$ is compactly generated and topologically simple. 
\end{prop}

\begin{proof}
The first assertion is \cite[Proposition~5.2]{CaMo-decomp}. The second assertion follows from the first, together with the description of the structure of compactly generated just-non-compact tdlc groups from \cite[Theorem~E]{CaMo-decomp}. 	
\end{proof}

Note however that the just-non-compact quotient group afforded by Proposition~\ref{prop:JNC-quotients} can in general be discrete. In fact, every compactly generated tdlc group with an infinite discrete quotient admits a just-non-compact discrete quotient (because every finitely generated infinite group admits a just-infinite quotient). In order to deal with discrete quotients, we need additional terminology.

Given a locally compact group $G$, the \textbf{discrete residual} of $G$, denoted by $\Res(G)$, is defined as the intersection of all open normal subgroups of $G$. Notice that $\Res(G)$ is a closed topologically characteristic subgroup of $G$. We say that $G$ is \textbf{residually discrete} if $\Res(G)=1$. The following result is due to G.\ Willis \cite{Willis-nilpotent}.

\begin{prop}\label{prop:VirtuallyNilp}
Let $G$ be a compactly generated tdlc group. If $G$ is virtually nilpotent, then $G$ has a basis of identity neighbourhoods consisting of compact open normal subgroups. In particular $G$ is residually discrete and compact-by-discrete. 
\end{prop}

The goal for the rest of this section is to establish Theorem~\ref{thm:MonolithicQuotient}, which ensures the existence of a specific kind of quotient groups for compactly generated tdlc groups whose discrete residual is also compactly generated. This requires several intermediate results that are of independent interest.

\begin{prop}[{\cite[Cor.~4.1]{CaMo-decomp}}]\label{prop:SIN}
A compactly generated tdlc group is residually discrete if and only if its compact open normal subgroups form a basis of identity neighbourhoods.
\end{prop}

Recall that a locally compact group $G$ is \textbf{regionally elliptic} if every compact subset of $G$ generates a compact subgroup of $G$ (this property is sometimes also called \enquote{locally elliptic}, or \enquote{topologically locally finite}. We refer to \cite{CRW_dense} for the motivation for the present choice of terminology). The \textbf{regionally elliptic radical} (RE-radical) of $G$ is the largest closed normal subgroup of $G$ that is regionally elliptic. 

\begin{prop}\label{prop:Res(G)-cptly-gen}
Let $G$ be a tdlc group.  If  $G$ and $\Res(G)$ are both compactly generated, then the following assertions hold, where $R_0 = G$ and for all $n \geq 0$, the group $R_{n+1}$ is defined as $R_{n+1} = \Res(R_n)$. 
\begin{enumerate}[label=(\roman*)]
\item $G/R_1$ is compact-by-discrete.
\item $R_1/R_2$ is compact. In particular $G/R_2$ is compact-by-discrete.
\item $R_n = R_2$ for all $n \geq 2$. 
\item $R_n$ is compactly generated for all $n$. 
\end{enumerate}
\end{prop}

\begin{proof}
By hypothesis, the groups $R_0$ and $R_1$ are compactly generated. Moreover the group $R_0/R_1$ is residually discrete, so that by Proposition~\ref{prop:SIN}, it is compact-by-discrete. Hence there exists an open normal subgroup $N_1$ of $G$ containing $R_1$ and such that $N_1/R_1$ is compact. 

Since $R_1$ is compactly generated, the group $N_1$ is also compactly generated. Moreover the quotient  $R_1/R_2$ is compactly generated and residually discrete. It then follows from Proposition~\ref{prop:SIN} that $R_1/R_2$ has a compact open normal subgroup. In particular, there exists a closed normal subgroup $N_2$ of $G$ with $R_2 \leq N_2 \leq R_1$ such that $N_2/R_2$ is the RE-radical of $R_1/R_2$. Moreover $N_2/R_2$ is open in $R_1/R_2$. Hence $R_1/N_2$ is discrete, so that the quotient group $N_1/N_2$ is compactly generated and discrete-by-compact. Therefore $N_1/N_2$ is compact-by-discrete (see \cite[Lemma 4.4]{BCGM-amen}). It follows that $G/N_2$ is compact-by-discrete. In particular $\Res(G/N_2)$ is compact. Since $N_2 \leq \Res(G) = R_1$, we have $\Res(G/N_2) = \Res(G)/N_2$. Therefore $R_1/N_2$ is compact. Since $R_1$ is compactly generated, we infer that $N_2$ is compactly generated as well. It follows that the regionally elliptic group $N_2/R_2$ is compactly generated, hence compact. Therefore $R_1/R_2$ is compact, and $R_2$ is also compactly generated. 

It remains to show that $R_3 = R_2$. We know that $R_2/R_3$ is compactly generated and residually discrete. Repeating the arguments of the previous paragraph, we construct a closed normal subgroup $N_3$ of $G$ with $R_3 \leq N_3 \leq R_2$ and such that $N_3/R_3$ is the RE-radical of $R_2/R_3$. Since $R_2/N_3$ is discrete and since $N_1/R_1$ and $R_1/R_2$ are both compact, we infer that $N_1/N_3$ is compactly generated and discrete-by-compact. Therefore $N_1/N_3$ is compact-by-discrete. As above, this implies that $R_1/N_3$ is compact, so that $R_2/N_3$ is also compact, and $N_3$ is compactly generated. This implies in turn that the regionally elliptic group $N_3/R_3$ is compactly generated, hence compact. This finally implies that $R_2/R_3$ is compact. Therefore $R_1/R_3$ is compact, hence profinite, so that $\Res(R_1/R_3) = 1$. Since $R_3 \leq R_2 = \Res(R_1)$, we deduce that $\Res(R_1)/R_3 = 1$. In other words, we have $R_2 = R_3$, as required. 
\end{proof}

A special situation where the hypotheses of Proposition~\ref{prop:Res(G)-cptly-gen} are satisfied is described in the following. 

\begin{cor}\label{cor:SmallOpenNorm}
Let $G$ be a tdlc group all of whose open normal subgroups are compactly generated. Then $G$ and $\Res(G)$ are compactly generated; in particular the conclusions of Proposition~\ref{prop:Res(G)-cptly-gen} hold. 
\end{cor}

\begin{proof}
Clearly $G$ is compactly generated by hypothesis. Since $G/\Res(G)$ is residually discrete, Proposition~\ref{prop:SIN} ensures that there exists an open normal subgroup $N$ of $G$ containing $\Res(G)$ and such that $N/\Res(G)$ is compact. In particular $\Res(G)$ is compactly generated, since $N$ has this property by hypothesis. 
\end{proof}

We shall prove the following result, which is a slight strengthening of \cite[Proposition~II.1]{CaMo-decomp}. A normal subgroup $N$ of $G$ is \textbf{maximal} if $N$ is proper and every normal subgroup $M$ of $G$ containing $N$ is equal to $N$ or $G$; and $N$ is \textbf{minimal} if $N$ is non-trivial and every normal subgroup $M$ of $G$ contained in $N$ is equal to $1$ or $N$. 

\begin{prop}\label{prop:UpperStr}
Let $G$ be a non-trivial compactly generated tdlc group with $\Res(G) = G$. Let $\mathbf{Max}$ (resp. $\mathbf{Min}$) be the collection of maximal (minimal) closed normal subgroups of $G$. Assume that  $\bigcap \mathbf{Max}= 1$. Then:
\begin{enumerate}[label=(\roman*)]
\item\label{it:Upper1} $\mathbf{Min}$ and $\mathbf{Max}$ are finite and non-empty. 
\item\label{it:Upper2} The assignment $N \mapsto C_G(N)$ establishes a bijective correspondence from $\mathbf{Min}$ to $\mathbf{Max}$. Moreover $N$ is the unique element of $\mathbf{Min}$ that is not contained in $C_G(N)$. 
\item\label{it:Upper3} The product of all elements of  $\mathbf{Min}$ is dense in $G$.
\item\label{it:Upper4} Every element of $\mathbf{Min}$ is a non-discrete compactly generated topologically simple tdlc group. 
\end{enumerate}
\end{prop}
\begin{proof}
The hypotheses imply that $\mathbf{Max}$ is non-empty. Moreover, for every $M \in \mathbf{Max}$, the quotient $G/M$ is topologically simple and non-discrete since $G = \Res(G)$. 

We claim that $G$ has a trivial center. Indeed, if $z$ were a non-trivial element in $Z(G)$, then there would exist $M \in \mathbf{Max}$ not containing $z$, because  $\bigcap \mathbf{Max}= 1$. The image of $z$ in $G/M$ would be a non-trivial central element, contradicting that $G/M$ is topologically simple and non-discrete (in particular infinite, thus non-abelian). 

All the hypotheses of Proposition~II.1 in  \cite{CaMo-decomp} are thus fulfilled. We deduce that the assertions \ref{it:Upper1}, \ref{it:Upper3}  and the first assertion of  \ref{it:Upper2} hold.  

Let $N \in \mathbf{Min}$. 
Notice first that $N$ is non-abelian. Indeed, since  $\bigcap \mathbf{Max}= 1$ there exists $M \in \mathbf{Max}$ with $N \not \leq M$. Then $N \cap M = 1$ since $N$ is minimal, and hence $M \leq C_G(N)$. Therefore either $M = C_G(N)$ or $G = C_G(N)$. The latter case is excluded, since $G$ has a trivial center as remarked above. Thus $M = C_G(N)$, and $N$ is non-abelian, as claimed. Moreover $C_G(N) \in \mathbf{Max}$, so that the projection map $N \to G/C_G(N)$ is an injective homomorphism with dense image on a compactly generated non-discrete topologically simple tdlc group. As observed in the proof of Proposition~II.1 in  \cite{CaMo-decomp}, we deduce that  $N$ has trivial quasi-center and trivial RE-radical. In particular the group $N$ is non-compact and non-discrete. 

Any  $N' \in \mathbf{Min}$ not contained in $M=C_G(N)$ commutes with $M$. If there existed such a group $N'$ distinct from $N$, it would also commute with $N$. Thus $N$ and $N'$ would be two commuting subgroups of $G$, both of which embed as dense normal  subgroup of $G/M$. This is impossible, because a Hausdorff group containing two commuting dense subgroups must be abelian, whereas $G/M$ is topologically simple and non-discrete, hence non-abelian. This confirms that $N$ is the unique element of $\mathbf{Min}$ that is not contained in $C_G(N)$.

Let $\varphi \colon G \to G/M$ be the canonical projection, let $U$ be a compact open subgroup of $G$ and let $V = N \cap U$. Then $\varphi(U)$ is a compact open subgroup of $G/M$, and $\varphi(V)$ is a closed normal subgroup of $\varphi(U)$. Since $N$ is a  minimal normal subgroup of $G$, it is generated by the $G$-conjugacy class of $V$. Equivalently, the group $\varphi(N)$ is generated by the $(G/M)$-conjugacy class of $\varphi(V)$. By Proposition~4.1(ii) in \cite{CRW-part2}, there is a finite subset $B \subset \varphi(N)$ such that $G/M = \langle B \rangle \varphi(U)$. Since $\varphi(U)$ normalizes $\varphi(V)$, it follows that the group $\langle B \rangle $ acts transitively on the $(G/M)$-conjugacy class of $\varphi(V)$. Therefore $\varphi(N)$ is generated by $B \cup \varphi(V)$. Equivalently $N$ is generated by $V \cup (\varphi^{-1}(B) \cap N)$, which is compact. Thus $N$ is compactly generated since $\varphi |_N$ is injective. 

Since $N \in \mathbf{Min}$, it is topologically characteristically simple. We may therefore invoke \cite[Corollary~D]{CaMo-decomp}, ensuring that $N$ has closed normal subgroups $S_1, \dots, S_k$ which are topologically simple, and such that $N = \overline{S_1\dots S_k}$. For each $i$, observe that $\varphi(S_i)$ is a non-trivial subgroup of $G/M$ normalized by $\varphi(N)$, which is dense. Since $G/M$ is topologically simple, we infer that $\varphi(S_i)$ is dense. Since $S_i \cap S_j = 1$ for $i \neq j$, we see that $S_i$ and $S_j$ commute. Since a Hausdorff group containing two commuting dense subgroups must be abelian, we deduce that $k=1$. Therefore $N$ is topologically simple. 
\end{proof}

\begin{rmq}\label{rem:Char-simple}
Let $G$ be a compactly generated, topologically characteristically simple tdlc group. If $G$ is non-compact and non-discrete, then Proposition~\ref{prop:SIN} implies that  $\Res(G) = G$. Moreover \cite[Theorem~A]{CaMo-decomp} ensures that the set $\mathbf{Max}$  of maximal proper closed normal subgroups of $G$ is non-empty. Hence $\bigcap \mathbf{Max}$ is trivial since $G$ is topologically characteristically simple. Therefore Proposition~\ref{prop:UpperStr} applies to $G$. This refines slightly the description in \cite[Corollary~D(iv)]{CaMo-decomp} by ensuring that the minimal normal subgroups of $G$ are compactly generated. 
\end{rmq}

A locally compact group $G$ is \textbf{monolithic} if the intersection $M$ of all non-trivial closed normal subgroups of $G$ is non-trivial. In that case $M$ is called the \textbf{monolith} of $G$. We record the following consequence of Proposition~\ref{prop:UpperStr}. The fact that a non-discrete compactly generated just-non-compact tdlc group is monolithic was established in \cite[Theorem~E(i)]{CaMo-decomp}. The additional conclusions in Proposition~\ref{prop:UpperStr-monolith}(ii) strengthen the latter result. 

\begin{prop} \label{prop:UpperStr-monolith}
Let $G$ be a compactly generated  tdlc group. 
\begin{enumerate}[label=(\roman*)]
\item If $G$ is monolithic and its monolith $M$ of $G$ is non-discrete, non-compact and compactly generated, then $M$ satisfies the conclusions of Proposition~\ref{prop:UpperStr}. 
\item If $G$ is non-discrete and just-non-compact, then it is monolithic, and its  monolith satisfies the conclusions of Proposition~\ref{prop:UpperStr}. 
\end{enumerate}
\end{prop}
\begin{proof}
By definition, the monolith is  topologically characteristically simple. Hence we may invoke Proposition~\ref{prop:UpperStr}, see Remark~\ref{rem:Char-simple}. This proves (i). Assume now that $G$ is non-discrete and just-non-compact. Then the assumptions of (i) are fulfilled by \cite[Theorem~E(i)]{CaMo-decomp}. The conclusion follows. 
\end{proof}

The following observation shows that monolithic groups arise naturally as quotients of groups admitting minimal normal subgroups.

\begin{lem}\label{lem:MonolithicQuotient}
Let $G$ be a locally compact group and $N$ be a minimal non-trivial normal subgroup of $G$. If $N$ is non-abelian, then $G/C_G(N)$ is monolithic, with monolith $\overline{C_G(N)N}/C_G(N)$. 
\end{lem}

\begin{proof}
(Compare the proof of Corollary~3.3 in \cite{CRW-part2}.) Since $N$ is non-abelian, it is not contained in $C_G(N)$, so that $\overline{C_G(N)N}/C_G(N)$ is a non-trivial closed normal subgroup of $G/C_G(N)$. In order to complete the proof, we must show that every closed normal subgroup $M$ of $G$ containing  $C_G(N)$ as a proper subgroup, also contains $N$. If the inclusion $C_G(N) < M$ is strict, then $M$ does not commute with $N$ and, hence $M \cap N \neq 1$. Since $N$ is minimal, it follows that $N \leq M$, as required. 
\end{proof}

\begin{thm}\label{thm:MonolithicQuotient}
Let $G$ be a compactly generated tdlc group such that $\Res(G)$ is also compactly generated. If $G$ is not compact-by-discrete, then $G$ has a closed normal subgroup $N$ such that the quotient $H = G/N$ enjoys the following properties. 
\begin{enumerate}[label=(\roman*)]
\item $H$ is monolithic, and its monolith $M$ is compactly generated, non-compact and non-discrete. 
\item $M$ has closed normal subgroups $S_1, \dots, S_k$ which are compactly generated, topologically simple and non-discrete, and such that $M = \overline{S_1\dots S_k}$.
\item $H/M$ is compact-by-discrete. 
\end{enumerate}
\end{thm}

\begin{proof}
We start by invoking Proposition~\ref{prop:Res(G)-cptly-gen}. This ensures that $R = \Res(\Res(G))$ is a compactly generated closed normal subgroup of $G$ such that $G/R$ is compact-by-discrete and $\Res(R) = R$. Since $G$ is not compact-by-discrete, we deduce that $R$ is non-trivial. 

It follows from \cite[Theorem~A]{CaMo-decomp} that $R$ admits a maximal proper closed normal subgroup $M_R$. Let $L = \bigcap_{g \in G} gM_Rg^{-1}$. Set $G' = G/L$, $M' = M_R/L$ and $R' = R/L$. Then $R' = \Res(\Res(G'))$ and $\Res(R') = R'$. Moreover $G'/R' \cong G/R$ is compact-by-discrete. Furthermore, by the definition of $R'$, we see that the intersection of all maximal proper closed normal subgroups of $R'$ is trivial. It follows that the structure of $R'$ is subjected to Proposition~\ref{prop:UpperStr}. By the definition of $G'$, the intersection of all $G'$-conjugates of $M'$ is trivial. In view of Proposition~\ref{prop:UpperStr}\ref{it:Upper2}, it follows that every maximal proper closed normal subgroup of $R'$ is conjugate to $M'$ in $G'$. Since every proper closed normal subgroup of $R'$ is contained in a maximal one (by Zorn's lemma, using that $R'$ is compactly generated), it follows that the only closed normal subgroup of $G'$ that is properly contained in $R'$ is the trivial subgroup. In other words, we have established that $R'$ is a minimal non-trivial closed normal subgroup of $G'$. 

We finally set $H = G'/C_{G'}(R')$. Thus $H$ is monolithic with monolith $M:=\overline{C_{G'}(R')R'}/C_{G'}(R')$ by Lemma~\ref{lem:MonolithicQuotient}. Since 
$$H/M = G'/C_{G'}(R') \bigg/ \overline{C_{G'}(R')R'}/C_{G'}(R') \cong G'/ \overline{C_{G'}(R')R'},$$
we infer that $H/M$ is isomorphic to a quotient of $G'/R' $. Since $G'/R' $ is compact-by-discrete, and since every Hausdorff quotient of a compact-by-discrete group is itself compact-by-discrete, we deduce that $H/M$ is compact-by-discrete. Since $R'$ is a quotient of $R$, it is compactly generated. Since it maps onto a dense subgroup of $M$, it follows that $M$ is compactly generated. Since $\Res(R') = R'$, we deduce that $\Res(M) = M$. In particular $M$ possesses maximal proper closed normal subgroups by Theorem~A in \cite{CaMo-decomp}. Since $M$ is minimal normal in $H$, the intersection of all its maximal proper closed normal subgroups is trivial, and the remaining assertion to be established follows from Proposition~\ref{prop:UpperStr}.
\end{proof}

\subsection{The centralizer lattice} \label{subsec-LCG}

Let $G$ be a tdlc group. A subgroup of $G$ whose normalizer is open is called \textbf{locally normal}. 
Following \cite{CRW-part1, CRW-part2}, we say that $G$ is \textbf{[A]-semisimple} if $\QZ(G)= \{1\}$ and if the only abelian locally normal subgroup of $G$ is the trivial subgroup. The notation $\QC_G(M)$ has been recalled in Section~\ref{sec:QZ}. 

\begin{prop}\label{prop:[A]-semisimple}
Let $G$ be a monolithic tdlc group, and assume that the monolith $M$ is compactly generated, non-compact and non-discrete. Then $G$ are $M$ are [A]-semisimple, and $\QC_G(M) = 1$.
\end{prop}

\begin{proof}
Since $M$ is a minimal normal subgroup of $G$, it is characteristically simple. Since $M$ is non-compact and non-discrete, it follows from \cite[Proposition~5.6]{CRW-part2} that $\QZ(M)=1$. We may then invoke \cite[Proposition~6.17]{CRW-part1}, which implies, together with \cite[Proposition~5.6]{CRW-part2}, that $M$ is  [A]-semisimple. 

The rest of the proof follows the reasoning in the proof of \cite[Proposition~5.1.2]{CRW_dense}. By \cite[Proposition~4.4.3]{CRW_dense}, we have $\QC_G(M) = C_G(M)$. Since $M$ is the monolith of $G$, either $C_G(M) = 1$ or $M \leq C_G(M)$. The second case is excluded, because  $\QZ(M)= 1$. Therefore $\QC_G(M) = 1$. In  particular $\QZ(G) = 1$. 

Let now $A$ be an abelian compact locally normal subgroup of $G$. Let $U$ be a compact open subgroup of $N_G(A)$. Upon replacing $U$ by $AU$, we may assume that $A$ is contained in $U$. Since $A \cap M$ is an abelian compact locally normal subgroup of $M$, and since $M$ is  [A]-semisimple, we have $A \cap M = 1$. In particular $A \cap M \cap U = 1$. Thus $A$ and $M \cap U$ are normal subgroups of $U$ with trivial intersection, hence they commute. It follows that $A$ centralizes an open subgroup of $M$. Hence $A \leq \QC_G(M) = 1$. 
\end{proof}

Two closed subgroups $K, L$ of a tdlc group $G$ are called \textbf{locally equivalent} if $K \cap L$ is relatively open in both $K$ and $L$. 
Following \cite{CRW-part1}, we define the \textbf{structure lattice} $\LN(G)$ of a tdlc group $G$ as the set  of local classes of closed   locally normal subgroups of $G$. The local class of a locally normal subgroup $K$ will be denoted $[K]$. Note that the group $G$ acts by conjugation on $\LN(G)$. The structure lattice $\LN(G)$ contains a canonical subset $\LC(G)$, called the \textbf{centralizer lattice}, consisting of the local classes of centralizers of locally normal subgroups of $G$. It is shown in \cite{CRW-part1} that, if $G$ is  [A]-semisimple, then the map $^\perp: \LC(G) \to \LC(G)$, $[K]^\perp = [C_G(K)]$, is well-defined and the operations \[ [K] \wedge [L] = [K \cap L] \] and \[ [K] \vee [L] = \left( [K]^\perp \wedge  [L]^\perp  \right)^\perp  \] turn $\LC(G)$ into a Boolean algebra. In addition we have the following result. The Stone space of a Boolean algebra $\mathcal A$ will be denoted by $\mathfrak S(\mathcal A)$.

\begin{thm} \label{thm-LCG-MS}
If a tdlc group $G$ is  [A]-semisimple, then the centralizer lattice $\LC(G)$ is a Boolean algebra, and the following hold:
\begin{enumerate}[label=(\roman*)]
 \item If $G$ acts faithfully on $\LC(G)$, then the $G$-action on $\Omega_G = \mathfrak S(\LC(G))$ is continuous and micro-supported.
 \item  Every totally disconnected compact $G$-space on which the $G$-action is faithful and  micro-supported is a $G$-factor of $\Omega_G$.
\end{enumerate}
\end{thm}
\begin{proof}
This follows from Theorem 5.2 and Theorem 5.18 in \cite{CRW-part1} (see also Theorem~II in loc. cit.).
\end{proof}

The following result addresses the case where the $G$-action on $\LC(G)$ is not necessary faithful. We recall that every Boolean algebra $\mathcal A$ is naturally endowed with a partial order $\leq$, defined by $\alpha \leq \beta$ if $\alpha \wedge \beta = \alpha$ for all $\alpha, \beta \in \mathcal A$. The partial order $\leq$ coincides with the relation of inclusion of clopen subsets of the Stone space $\mathfrak S(\mathcal A)$. Given a subset $H$ of a group $G$, we denote by $C_G^2(H)$ the \textbf{double centralizer} of $H$ in $G$, defined by $C_G^2(H) = C_G(C_G(H))$.

\begin{prop}\label{prop:FaithfulDenseSubgroup}
Let $G$ be an [A]-semisimple tdlc group, $D \leq G$ be a dense subgroup and $Y$ be a totally disconnected compact $G$-space. Assume that the $D$-action on $Y$ is micro-supported and faithful. 

Then the assignment $\alpha \mapsto [C_G^2(\rist_D(\alpha))]$ defines an injective order-preserving $G$-equivariant map from the Boolean algebra $\mathcal A$ of clopen subsets of $Y$ into the centralizer lattice $\LC(G)$.
\end{prop}

\begin{proof}
Let $\alpha \in \mathcal A$. Since the $G$-action on $Y$ is continuous, the stabilizer $G_\alpha$ is open in $G$. In particular $D_\alpha = D \cap G_\alpha$ is dense in $G_\alpha$. Since $D_\alpha$ normalizes $\rist_D(\alpha)$, we infer that $\overline{\rist_D(\alpha)}$ is normalized by $G_\alpha$. In particular it is a locally normal subgroup of $G$. 

Therefore $C_G(\overline{\rist_D(\alpha)}) = C_G( {\rist_D(\alpha)}) $ is a closed locally normal subgroup as well, and so is also $C_G^2( {\rist_D(\alpha)}) $. By the definition of $\LC(G)$ (see \cite[Definition~5.1]{CRW-part1}), the local classes $[C_G( {\rist_D(\alpha)})]$ and $[C_G^2( {\rist_D(\alpha)})] $ both belong to $\LC(G)$. 

For $\alpha \in \mathcal A$, we set
$$f(\alpha) := [C_G^2( {\rist_D(\alpha)})] \in \LC(G).$$
We claim that if $\alpha$ is non-empty, then $f(\alpha)$ is   a non-zero element of $\LC(G)$.  If $f(\alpha)=0$, then $C_G^2( {\rist_D(\alpha)})$ is discrete, and thus contained in the quasi-center $\QZ(G)$, which is trivial since $G$ is [A]-semisimple by hypothesis. Since $C_G^2( {\rist_D(\alpha)})$ contains $\overline{\rist_D(\alpha)}$, we deduce that $\rist_D(\alpha)$ is trivial, hence $\alpha = \varnothing$ since the $D$-action on $Y$ is micro-supported. This proves the claim. 

Let now $\alpha, \beta \in \mathcal A$ with $\alpha \leq \beta$. Then $\rist_D(\alpha) \leq \rist_D(\beta)$. Therefore we have $C_G^2( {\rist_D(\alpha)}) \leq C_G^2( {\rist_D(\beta)})$ and hence $f(\alpha) \leq f(\beta)$. It follows that the map $f \colon \mathcal A \to \LC(G)$    is order-preserving. This implies in particular that for all $\beta, \gamma \in \mathcal A$, we have $f(\beta \cap \gamma) \leq f(\beta ) \wedge  f(\gamma )$. 

We next claim that if $\beta \cap \gamma = \varnothing$, then $f(\beta ) \wedge  f(\gamma ) = 0$. Indeed, if $\beta$ and $\gamma$ are disjoint, then $\rist_D(\beta)$ and $\rist_D(\gamma)$ commute. Thus $\rist_D(\beta) \leq C_G(\rist_D(\gamma))$, hence $C_G^2(\rist_D(\beta)) \leq C_G(C_G^2(\rist_D(\gamma)))$. Thus the closed locally normal subgroups $C_G^2(\rist_D(\beta))$ and $C_G^2(\rist_D(\gamma))$ commute. In view of \cite[Theorem~3.19(iii)]{CRW-part1}, we deduce that $f(\beta ) \wedge  f(\gamma ) = 0$.

It remains to prove the injectivity of $f$. To this end, let $\alpha, \beta \in \mathcal A$ with $\alpha \neq \beta$. Upon swapping $\alpha$ and $\beta$, we may assume that $\alpha \not \leq \beta$. Equivalently there exists a non-empty $\gamma \leq \alpha$ in $\mathcal A$ which is disjoint from $\beta$. Since $\beta \cap \gamma = \varnothing$, we have $f(\beta ) \wedge  f(\gamma ) = 0$ by the previous claim. Since $\gamma \leq \alpha$ we have $f(\gamma ) \leq f(\alpha )$. If we had $f(\alpha )=f(\beta)$, we would have $ f(\gamma )=  f(\gamma )  \wedge  f(\alpha ) = f(\gamma ) \wedge f(\beta)= 0$. Since $\gamma$ is non-empty, this would contradict the claim established above. Thus  $f(\alpha ) \neq f(\beta )$, as required.
\end{proof}

\subsection{Dynamical features of micro-supported actions of non-discrete groups}

The dynamics of the action of a compactly generated  [A]-semisimple group $G$ on the Stone space $\Omega_G$ of the centralizer lattice of $G$ has been studied in \cite[Theorem~6.19]{CRW-part2} under the assumption that the $G$-action on $\Omega_G$ is faithful.

Recall from Proposition~\ref{prop:[A]-semisimple} that a compactly generated monolithic tdlc group $G$ whose monolith is compactly generated, non-compact and non-discrete is [A]-semisimple. In particular the centralizer lattice $\LC(G)$ is a Boolean algebra. Its Stone space, denoted by $\Omega_G$, is thus a totally disconnected compact $G$-space. 

\begin{thm}\label{thm:Stone-dynamics}
Let $G$ be a compactly generated monolithic tdlc group,  and assume that the monolith $M$ is compactly generated, non-compact and non-discrete. Let $\Omega_G$ be the Stone space of $\LC(G)$. 
 
\begin{enumerate}[label=(\roman*)]
	\item If $\LC(G)$ is infinite, then the $G$-action on $\Omega_G$ is faithful.
\end{enumerate}
 
 Moreover, any  totally disconnected compact $G$-space $X$  on which the $G$-action is faithful and micro-supported is a factor of $\Omega_G$, and enjoys the following properties:
\begin{enumerate}[resume*]	
\item \label{it:Dyn1} The $M$-action on $X$ is micro-supported. 

\item \label{it:Dyn2}  The $G$-action on $X$ has a compressible clopen subset.

\item \label{it:Dyn3} The $G$-action on $X$ is minimal. 

\item \label{it:Dyn4} Let $S_1, \dots, S_d$ be the minimal non-trivial closed normal subgroups of $M$. For each $i$, let $X_i$ be the complement of the fixed-point-set of $S_i$ in $X$. Then $X_i$ is clopen and $X = \bigcup_{i=1}^d X_i$ is a clopen partition of $X$. Furthermore, the $S_i$-action on $X_i$ is minimal, strongly proximal, micro-supported, and has a compressible clopen subset. In particular $M$ is not amenable.
\end{enumerate}
\end{thm}


\begin{proof}
	Recall that the fact the  $S_i$ exist and form a finite set is guaranteed by Proposition \ref{prop:UpperStr-monolith}. Moreover each $S_i$ is a non-discrete compactly generated topologically simple group, and $S_1 \cdots S_d$ is a dense subgroup of $M$.
	
We argue by contradiction and assume that the $G$-action on $\Omega_G$ is not faithful. In particular $M$ acts trivially on $\Omega_G$.

For each $i$, we define  
$$\omega_i \colon \LC(G) \to \{0, 1\} : [K] \mapsto \left\{
\begin{array}{ll}
1 & \text{if } S_i \cap  C_M(K) = 1\\
0 & \text{otherwise.}
\end{array}
\right.$$

Since $M$ acts trivially on $\LC(G)$, for every $[K] \in \LC(G)$ the group $\QC_M(K)$ is equal to $C_M(K)$ and is a closed normal subgroup of $M$ by \cite[Theorem~3.19]{CRW-part1}. In particular since $S_i$ is a minimal normal subgroup  of $M$, either $S_i \leq C_M(K)$ or  $S_i$ intersects $C_M(K)$ trivially. So $\omega_i([K]) = 0$ if and only if $S_i$ commutes with $K$, if and only if $S_i \cap K = 1$ \cite[Theorem~3.19]{CRW-part1}. 

We claim that $\omega_i$ is a homomorphism of Boolean algebras. The previous paragraph implies that  $\omega_i([K]^\perp) = \omega_i([K])^\perp$. Hence it is enough to check that for all $[H], [K] \in \LC(G)$, $\omega_i([H] \wedge [ K]) = \omega_i([H\cap K])$ is equal to $\omega_i([H]) \wedge \omega_i([K])$. Clearly if $\omega_i([H]) = 0$ or $\omega_i([K]) = 0$, i.e. if $S_i$ commutes with $H$ or $S_i$ commutes with $K$, then $S_i$ commutes with $H \cap K$ and $\omega_i([H\cap K])=0$. Moreover if $\omega_i([H]) = \omega_i([K]) = 1$, then $S_i \cap H \cap K \neq 1$. Indeed otherwise $S_i \cap H$ would commute with $K$ by \cite[Theorem~3.19]{CRW-part1}, and hence $S_i \cap H$ would be abelian because by assumption $S_i$ commutes with $C_G(K)$ since $\omega_i([K]) = 1$. Hence $S_i \cap H$ would be an abelian locally normal subgroup of $S_i$, which is non-trivial because $\omega_i([H]) =1$. This contradicts \cite[Theorem 5.3]{CRW-part2}. Hence $S_i \cap H \cap K \neq 1$, and $\omega_i([H\cap K])=1$. This shows the claim. 

The claim implies that $\omega_i$ represents a point in the Stone space of $\LC(G)$, which is $\Omega_G$. 
For $[K] \in \LC(G)$, $K \neq 1$, the subgroup $K$ cannot commute will all the $S_i$, because otherwise $K$ would commute in $M$ and $M$ has trivial centralizer. Hence there exists $i$ such that $\omega_i([K]) = 1$. This means that every non-empty clopen subset of $\Omega_G$ contains a point $\omega_i$ for some $i$. It follows that $\Omega_G = \{\omega_1, \dots, \omega_d\}$. This implies that $\LC(G)$ is finite, a contradiction. Therefore,  the $G$-action  on $\Omega_G$ is indeed faithful.

Let now $X$ be a  totally disconnected compact $G$-space   on which the $G$-action is faithful and micro-supported. By Theorem \ref{thm-LCG-MS}, the space $X$ is a factor of $\Omega_G$. Equivalently, the centralizer lattice $\LC(G)$ contains a $G$-invariant Boolean subalgebra $\mathcal A$ whose Stone dual is $G$-equivariantly homeomorphic to $X$. We may thus identify the clopen subsets of $X$ with the elements of $\mathcal A$. Since the $G$-action on $X$ is faithful, so is the action on $\mathcal A$. 

For every non-empty clopen $\alpha$ of $X$, the rigid stabilizer $\rist_G(\alpha)$ is a non-discrete closed locally normal subgroup of $G$. In view of Proposition~\ref{prop:[A]-semisimple}, we may invoke \cite[Proposition~7.1.2(i)]{CRW_dense} which ensures that $\rist_G(\alpha) \cap M = \rist_M(\alpha)$ is non-discrete. This proves \ref{it:Dyn1}.

By \cite[Proposition~7.3.1]{CRW_dense}, the $G$-action on $X$ has a compressible clopen subset, thus \ref{it:Dyn2} holds. 

To prove \ref{it:Dyn3} and \ref{it:Dyn4},  we let $X_i$ be the complement of the fixed-point-set of $S_i$ in $X$, for each $i \in \{1, \dots, d\}$. Since the $G$-action on $X$ is faithful, so is the $M$-action, and we deduce that $X_i$ is non-empty. Moreover $M$ is  [A]-semisimple by Proposition~\ref{prop:[A]-semisimple}. In view of \ref{it:Dyn1}, we may therefore  invoke \cite[Theorem~6.19]{CRW-part2} for the $M$-action on $X$ (the latter result requires that $M$ is \textit{locally C-stable}, which is a consequence of  [A]-semisimplicity by    \cite[Proposition~6.17]{CRW-part1}). This ensures that the 
$X_i$ are pairwise disjoint. 

We next claim that for each $i \in \{1, \dots, d\}$, the $S_i$-action on $\overline{X_i}$ is micro-supported and has a compressible clopen subset. Indeed, applying   \cite[Lemma~6.11 and Theorem~6.19(i)]{CRW-part2} to the $M$-action on $X$, we obtain,  for each $i$, a non-empty clopen $\alpha_i \subset X_i$ such that for every non-empty clopen $\beta$ in $X$, there exists $i \in \{1, \dots, d\}$ and $g \in M$ with $g\alpha_i \subset \beta$. 

Let us now fix some index $j \in \{1, \dots, d\}$ and a non-empty clopen $\beta$ is contained in $\overline{X_j}$. Notice that $\beta \cap X_k = \varnothing$ for all $
k \neq j$. By the previous paragraph there exists an index $i \in \{1, \dots, d\}$, a non-empty clopen subset $\alpha_i \subset X_i$  and an element  $g \in M$ with $g\alpha_i \subset \beta$. 
Since the $M$-action is continuous and since the product $S_1\dots S_d$ is dense in $M$, we may assume that $g \in S_1\dots S_d$. Let us write $g = g_1 \dots g_d$ with $g_k \in S_k$. In particular the elements $g_k$ commute pairwise. Since $\beta \cap X_k = \varnothing$ for all $
k \neq j$, it follows that $\beta$ is pointwise fixed by $g_k$ for all $k \neq j$. Therefore we have   $g_j \alpha_i \subset (\prod_{k \neq j} g_k^{-1}) \beta  = \beta$.  This implies that $i = j$, thereby   proving that $\alpha_j$ is a compressible clopen subset for the $S_j$-action on the closure $\overline{X_j}$. 

Notice that $\rist_M(\alpha_i)$ is a non-trivial closed locally normal subgroup of $M$. Using \cite[Theorem~6.19(iii)]{CRW-part2}, we deduce that  $\rist_{S_i}(\alpha_i)$ is non-discrete. Since $\alpha_i$ is a compressible clopen subset for the $S_i$-action on  $\overline{X_i}$, it follows that the $S_i$-action on  $\overline{X_i}$ is micro-supported. This proves the claim. 

Since $S_i$ is a compactly generated, non-discrete, topologically simple tdlc group, we may invoke \cite[Theorem~J(ii)]{CRW-part2}, which ensures that the $S_i$-action on  $\overline{X_i}$ is minimal and strongly proximal. In particular $S_i$ does not fix any point of $\overline{X_i}$. Therefore $X_i = \overline{X_i}$, thereby confirming that $X_i$ is clopen in $X$. Since the $M$-action on $X$ is micro-supported by \ref{it:Dyn2}, and since $M$ acts trivially on the open set $X \setminus (\bigcup_{i=1}^d X_i)$, we deduce that   $X = \bigcup_{i=1}^d X_i$. Thus \ref{it:Dyn4} holds. 

We finally observe that the clopen subsets $X_i$ are the minimal non-empty $M$-invariant closed subsets of $X$. In particular every non-empty $G$-invariant closed subset contains $X_i$ for some $i$. Since $M$ is the monolith of $G$, the $G$-action by conjugation permutes the set $\{S_1, \dots, S_d\}$ transitively. Therefore the $G$-action on $X$ permutes the sets $\{X_1, \dots, X_d\}$ transitively. In view of  \ref{it:Dyn4}, we deduce that \ref{it:Dyn3} holds.
\end{proof}

\begin{rmq}
We emphasize that the compact generation of the monolith is essential for assertion~\ref{it:Dyn3} in Theorem~\ref{thm:Stone-dynamics}. An example illustrating this matter of fact is provided by the group $G = \mathrm{Ner}(T)_\xi$, where $T$ is the $d$-regular tree, $d \geq 3$, $\xi$ is an end of $T$, and  $\mathrm{Ner}(T)$ is the Neretin group of $T$ (which acts on $\partial T$). The $G$-action on $\partial T$ is faithful and micro-supported. Moreover, the group $G$ is monolithic, but its monolith is not compactly generated. The $G$-action on $\partial T$ has a compressible clopen subset, but it is not minimal as it fixes $\xi$.
\end{rmq}

\section{Dense embeddings of groups with a micro-supported action} \label{sec-embed-micro}

In the sequel when $\Gamma$ is a group acting on a topological space $X$, we will denote by $\Gamma_x$ the stabilizer of $x \in X$ in $\Gamma$, and by $\Gamma_x^0$ the elements $\gamma \in \Gamma$ such that $\gamma$ fixes pointwise a neighbourhood of $x$ in $X$. Observe that $\Gamma_x^0$ is a normal subgroup of $\Gamma_x$. 

\subsection{The general setting} \label{sec:GeneralSetting}

The setting of the current section will be the following:

\begin{enumerate}[label=\alph*)]
	\item $L$ is a locally compact group, and $X$ is a compact $L$-space such that the action of $L$ on $X$ is faithful, minimal and micro-supported.
	\item  $\varphi: L \to G$ is a continuous injective homomorphism from $L$ into a second countable tdlc group $G$, such that $\varphi(L)$ is dense in $G$.
\end{enumerate}

We emphasize that, while the group $G$ is typically non-discrete, the group $L$ is allowed to be discrete. While the applications in Section \ref{sec-applic}  that will be obtained in \S \ref{subsec-full-gp}-\ref{subsec-branch}-\ref{subsec-circle} only deal with the case where $L$ is a discrete group, the setting of \S \ref{subsec-dense-LC} requires $L$ to be non-discrete, whence this choice for the current section. Notice that if the group $L$ is generated by a subset $S$, then $G$ is generated by $\varphi(S)$ together with any neighbourhood of the identity because $\varphi(L)$ is dense in $G$. In particular if $L$ is compactly generated then $G$ is compactly generated as well.

\subsection{Construction of a URS in $G$} \label{subsec-URS}

Let $G$ be a locally compact group, and denote by $\sub(G)$ the space of closed subgroups of $G$, endowed with the Chabauty topology. The space $\sub(G)$ is compact, and is metrizable  if $G$ is second countable. A  closed minimal $G$-invariant subspace of $\sub(G)$ is called a \textbf{uniformly recurrent subgroup} (URS) of $G$ \cite{GW-urs}. 

Let $X$ be a compact space. Recall that a map $\psi \colon X \to \sub(G)$ is \textbf{lower semi-continuous} if for every open subset $U \subset G$, the set of $x \in X$ such that $\psi(x) \cap U \neq \varnothing$ is open in $X$. This is equivalent to saying that for every net $(x_i)$ in $X$ converging to $x$ and such that the net $(\psi(x_i))$ converges to $H$, one has $\psi(x) \subseteq H$. 

In the sequel we consider $L,X,\varphi,G$ as in the general setting. Observe that the group $L$ has a continuous conjugation action on $\sub(G)$ defined via $\varphi$.


\begin{lem} \label{lem-Gammax0-lsc}
For $L,X,\varphi,G$ as in the general setting, the map $\psi : X \to \sub(G)$, $x \mapsto \overline{\varphi(L_x^0)}$, is lower semi-continuous and $L$-equivariant.
\end{lem}

\begin{proof}
Let $x \in X$ and $U$ an open subset of $G$ such that $\psi(x) \cap U \neq \varnothing$. Then $\varphi(L_x^0) \cap U \neq \varnothing$, and if $\gamma \in L_x^0$ is such that $\varphi(\gamma) \in U$, then by definition of $L_x^0$ there is an open neighbourhood $V$ of $x$ in $X$ on which $\gamma$ acts trivially. Then $\gamma \in L_y^0$ for every $y \in V$, and it follows that $\varphi(L_y^0) \cap U \neq \varnothing$ for every $y \in V$. This shows lower semi-continuity. That $\psi$ is $L$-equivariant follows from the definitions.
\end{proof}

In the sequel we  denote by $X_\varphi \subseteq X$ the set of points where the map $x \mapsto \overline{\varphi(L_x^0)}$ is continuous. Since the group $G$ is second countable, the space $\sub(G)$ is metrizable, hence by a general property of semi-continuity \cite[Th.~VII]{Kuratowski1928}, it follows from Lemma \ref{lem-Gammax0-lsc} that $X_\varphi$ is a dense subset of $X$.  In the sequel we write


\[ F_{\varphi}(X) := \overline{ \left\{ \left( x, \overline{\varphi(L_x^0)}\right)  \, : \, x \in X \right\}} \subseteq X \times \sub(G), \]

\[ E_{\varphi}(X) := \overline{ \left\{ \left( x, \overline{\varphi(L_x^0)}\right)  \, : \, x \in X_\varphi \right\}} \subseteq F_{\varphi}(X), \]

\[ T_{\varphi,G}(X) := \overline{ \left\{ \overline{\varphi(L_x^0)} \, : \, x \in X \right\}} \, \, \text{and} \, \, S_{\varphi,G}(X) := \overline{ \left\{ \overline{\varphi(L_x^0)} \, : \, x \in X_\varphi \right\}}. \]


Note that $L$ acts diagonally on  $X \times \sub(G)$, and $ F_{\varphi}(X)$ and $E_{\varphi}(X)$ are closed and $L$-invariant. In the following statement we denote respectively by $p_1,p_2$ the projections from $ X \times \sub(G)$ to the first and second factor. Recall that an extension $\pi: Y \to X$ between minimal compact $G$-spaces is \textbf{almost one-to-one} if the set of $y \in Y$ such that $\pi^{-1}(\pi(y)) = \{y\}$ is dense in $Y$. Note that every almost one-to-one extension is highly proximal.


\begin{prop} \label{prop-def-S-phi-G}
	The following hold:
	\begin{enumerate}[label=(\roman*)]
		\item $E_{\varphi}(X)$ is the unique minimal closed $L$-invariant subset of $F_{\varphi}(X)$, and $S_{\varphi,G}(X)$ is the unique minimal closed $G$-invariant subset of $T_{\varphi,G}(X)$. In particular  $S_{\varphi,G}(X)$ is a URS of $G$.
		\item The extension $p_1 : E_{\varphi}(X) \to X$ is almost one-to-one. Moreover, one has $p_2(E_{\varphi}(X)) = S_{\varphi,G}(X)$.
	\end{enumerate}

\end{prop}


\begin{proof}
See Theorem 2.3 in \cite{Glasner-compress}. The only additional observation that is needed here is that $S_{\varphi,G}(X)$ is indeed $G$-invariant, but this is clear since $S_{\varphi,G}(X)$ is $L$-invariant and $L$ has dense image in $G$.
\end{proof}


\subsection{Conditions ensuring that $S_{\varphi,G}(X)$ is infinite}

The goal of this section is to exhibit certain conditions ensuring that the space $S_{\varphi,G}(X)$ constructed in Section \ref{subsec-URS} is not degenerate (i.e.\ is not a finite set), and also conditions ensuring that the $G$-action  on $S_{\varphi,G}(X)$ is faithful.

The following key lemma is the starting point of our discussion. 

\begin{lem}\label{lem:Key}
	Let $\Gamma$ be a group with a faithful action on a Hausdorff space $X$. Let also $\varphi \colon \Gamma \to G$ be an injective homomorphism to a locally compact group $G$. 
	
	Let $x \in X$ and $H$ be a subgroup of the centralizer $C_G(\varphi(\Gamma_x^0))$. If the closed subgroup $J = \overline{\varphi(\Gamma_x^0)H} \leq G$ is compactly generated, then there is a non-empty open subset $V \subset X$ such that $\overline{\varphi(\rist_\Gamma(V))} \leq B(J)$.
\end{lem}

\begin{proof}
	Let us denote by $\mathcal{F}_x$ the family of closed subsets of $X$ not containing $x$. By definition one has 
	\[ \Gamma_x^0 = \bigcup_{C \in \mathcal{F}_x} \rist_\Gamma(C). \] 
	Since $\overline{\varphi(\Gamma_x^0)H} = J$, we deduce that $J$ belongs to the closure in $\sub(J)$ of the set 
	\[ \left\{ \overline{\varphi( \rist_\Gamma(C))H} \, : \, C \in \mathcal{F}_x \right\}. \] 
	By hypothesis, the group $J$ is compactly generated, and thus it admits a Chabauty neighbourhood consisting of cocompact subgroups (see \cite[VIII.5.3, Proposition 6]{Bourb-int-7-8}). Therefore we may find $C \in \mathcal{F}_x$ such that $\overline{\varphi( \rist_\Gamma(C))H}$ is cocompact in $J$. Let $U$ be the complement of $C$ in $X$. Since $X$ is Hausdorff, we may find a non-empty open subset $V \subset U$ such that  $U \setminus V$ is a neighbourhood of $x$. Then  $\rist_\Gamma(V)$ centralizes $\rist_\Gamma(C)$ because $V$ and $C$ are disjoint. Moreover, since $H$  centralizes $ \varphi(\Gamma_x^0)$ by hypothesis, and since $\varphi(\rist_\Gamma(V)) \leq \varphi(\Gamma_x^0)$, we deduce that $\overline{\varphi(\rist_\Gamma(V))}$ centralizes $\overline{\varphi( \rist_\Gamma(C))H}$. Since the latter is cocompact in $J$, it follows that $\overline{\varphi(\rist_\Gamma(V))} \leq B(J)$.
\end{proof}

\begin{prop} \label{prop-urs-phi-nontrivial-1}
	Let $L,X,\varphi,G$ as in the general setting, and assume that $G$ is compactly generated. Suppose also that the action of $L$ on $X$ has the property that for every closed subgroup $H \leq L$ such that the normalizer of $H$ in $L$ has finite index and such that there exists $x \in X$ with $L_x^0 \leq H$, then $H$ must be cocompact in $L$.

	If $S_{\varphi,G}(X)$ is finite, then there exists a non-empty open subset $V \subset X$ such that $\overline{\varphi(\rist_L(V))} \leq B(G)$. 
\end{prop}

\begin{proof}
By assumption there exists a closed subgroup $J \leq G$ such that $N_G(J)$ has finite index in $G$ and \[ S_{\varphi,G}(X) = \left\{J,g_1Jg_1^{-1}, \ldots,g_n J g_n^{-1}\right\} \] is the conjugacy class of $J$ in $G$. Let $H = \varphi^{-1}(J)$. By definition of $S_{\varphi,G}(X)$ we may find $x \in X$ such that $\overline{\varphi(L_x^0)} = J$, so in particular $L_x^0 \leq H$. Moreover the normalizer of $H$ in $L$ has finite index in $L$, so it follows from our assumption that $H$ must be cocompact in $L$. Since $L$ has dense image in $G$, this implies that $\overline{\varphi(H)}$ is cocompact in $G$. It follows that $J$ is also cocompact in $G$, so in particular $J$ is compactly generated. The conclusion follows from Lemma~\ref{lem:Key} using the fact that $B(J) \leq B(G)$ because $J$ is cocompact in $G$.
\end{proof}

\begin{defin}
	A group $L$ is \textbf{just non virtually nilpotent} (j.n.v.n.) if $L$ is not virtually nilpotent and every proper quotient of $L$ is virtually nilpotent. 
\end{defin}

\begin{defin}
	The action of $L$ on a minimal compact $L$-space $X$ is \textbf{strongly just-infinite} if for every distinct points $x, y$ in $X$, the subgroup of $L$ generated by $L_x^0$ and $L_y^0$ has finite index in $L$.
\end{defin}

Note that this condition implies that for every non-injective $L$-map $X \to Y$ that is a factor, the space $Y$ must be finite. This explains the choice of terminology.

\begin{prop} \label{prop-urs-phi-nontrivial-2}
Let $L,X,\varphi,G$ as in the general setting. Assume that $L$ is j.n.v.n., that the action of $L$ on $X$ is strongly just-infinite, and that the group $G$ is compactly generated. If $S_{\varphi,G}(X)$ is finite, then $G$ is compact-by-discrete.
\end{prop}

\begin{proof}
By hypothesis $L$ is j.n.v.n., hence it is infinite. Since the $L$-action on $X$ is minimal and faithful, we infer that  every $L$-orbit on $X$ is infinite. Therefore, the hypothesis that the $L$-action on $X$ is strongly just-infinite implies that all the assumptions of Proposition~\ref{prop-urs-phi-nontrivial-1} are satisfied. Therefore, if $S_{\varphi,G}(X)$ is finite, then Proposition~\ref{prop-urs-phi-nontrivial-1} yields an element $\gamma \in L$ such that $\varphi(\gamma) \in B(G)$. Let $N = \overline{\langle \langle \varphi(\gamma) \rangle \rangle_G}$. Thus $N$ is a compactly generated closed normal subgroup of $G$ contained in $B(G)$. It follows from Theorem~\ref{thm-usakov} that $N$ admits a compact normal subgroup $K$ such that $N/K$ is discrete and torsion-free abelian. In particular $K$ is characteristic, and hence normal in $G$. Since $N \cap \varphi(L)$ is non-trivial, it follows that the image of $L$ in $G/N$ is virtually nilpotent, and hence that $G/N$ is residually discrete by Proposition~\ref{prop:VirtuallyNilp}. Using \cite[Proposition 2.2]{CLB19}, we know that a compactly generated tdlc group that is discrete-by-\{residually discrete\} must be residually discrete, and hence compact-by-discrete by Proposition~\ref{prop:SIN}. Applying this to $G/K$, we infer that $G/K$ is compact-by-discrete. Therefore $G$ also has this property, as desired. 
\end{proof}

\begin{prop}\label{prop:S-phi-G-Faithful}
	Let $L,X,\varphi,G$ as in the general setting. Assume that $L$ is j.n.v.n.\, and that $G$ is compactly generated, monolithic, with a non-discrete, non-compact, compactly generated monolith $M$, and such that $G/M$ is compacy-by-discrete. Then the action of $G$ on $S_{\varphi, G}(X)$ is faithful.
\end{prop}

\begin{proof}
	By Proposition \ref{prop:UpperStr-monolith}, the set of minimal normal subgroups of $M$ is finite and non-empty. Its elements are denoted by $S_1, \dots, S_d$. Thus $S_i$ is a compactly generated, non-discrete, topologically simple tdlc group. Moreover we have $B(M)=1$. 
	
	
	Suppose the action of $G$ on $S_{\varphi, G}(X)$ is not faithful. Then $M$ acts trivially on $S_{\varphi, G}(X)$. In the sequel we fix a point $x$ such that $\overline{\varphi(L_x^0)}$ belongs to  $S_{\varphi, G}(X)$. The subgroup $M \cap \overline{\varphi(L_x^0)}$ is normal in $M$ because $M$ normalizes  $\overline{\varphi(L_x^0)}$. Therefore by minimality of $S_1, \dots, S_d$, for each $i$ we have either $S_i \leq \overline{\varphi(L_x^0)}$ or $S_i \cap \overline{\varphi(L_x^0)} = 1$. Let $I $ be the set of those $i \in \{1, \dots , d\}$ such that $S_i \cap \overline{\varphi(L_x^0)} = 1$, and set $H = \langle \bigcup_{i \in I} S_i \rangle$. Observe that $H$ centralizes $\varphi(L_x^0)$, and that $M = \overline{S_1 \dots S_d}$ is entirely contained in $J = \overline{\varphi(L_x^0) H}$. 
	
By hypothesis, the group $G/M$ is compact-by-discrete. Since $G$ is not discrete and $L$ has dense image in $G$, every discrete quotient of $G$ is a proper quotient of $L$, and hence is finitely generated virtually nilpotent. Therefore the group $G/M$ is compactly generated and compact-by-\{discrete finitely generated virtually nilpotent\}. It follows that every closed subgroup of $G/M$ is compactly generated. In particular $J/M$ is compactly generated, and thus also $J$ since  $M$ is compactly generated. Lemma~\ref{lem:Key} therefore implies that $B(J)$ is non-trivial. Since $B(M)$ is trivial, we have $B(J) \cap M = 1$ so that $B(J) \leq C_G(M)$. It follows that $C_G(M)$ is non-trivial. Since $M$ is the monolith of $G$, we must have $M \leq C_G(M)$, so $M$ is abelian. In particular $M$ is compact-by-discrete, and we have already seen above that this is prevented by the hypotheses.
\end{proof}



\subsection{Existence of continuous extensions}

The proof of the following proposition is inspired by \cite[\S 4.1.3]{LBMB-sub-dyn} (see also \cite[Proposition 7.18]{MB-urs-full}).

\begin{prop} \label{prop-urs-to-X-Gamma-eq}
Let $L,X,\varphi,G$ as in the general setting, and assume that the action of $L$ on $X$ is strongly just-infinite, and that $S_{\varphi,G}(X)$ is infinite. Then the following hold:
\begin{enumerate}[label=(\roman*)]
	\item \label{item-factor-S-to-X} There exists a $L$-map $\psi: S_{\varphi,G}(X) \to X$, and $\psi$ is almost one-to-one.
	\item \label{item-action-extends} The action of $L$ on $X$ extends to a continuous $G$-action, and the map $\psi: S_{\varphi,G}(X) \to X$ is a $G$-map.
	\item \label{item-action-extends-faith} If the $G$-action on $S_{\varphi,G}(X)$ is faithful then the $G$-action on $X$ is faithful.
\end{enumerate} 
\end{prop}

\begin{proof}
\ref{item-factor-S-to-X} For $K \in S_{\varphi,G}(X)$, we wish to define $\psi(K)$ to be the unique $x \in X$ such that $\overline{\varphi(L_x^0)} \leq K$. In order to see that $\psi$ is well-defined, we have to show that such a point $x$ exists and is unique. Recall that there exists a dense subset of points $K \in S_{\varphi,G}(X)$ which are of the form $\overline{\varphi(L_x^0)}$ for some $x \in X$. So for these points the existence is clear, and the general case where $K$ is arbitrary follows by semi-continuity (Lemma~\ref{lem-Gammax0-lsc}). Suppose now that there is $K$ for which this point $x$ is not unique. Then by the assumption that the action of $L$ on $X$ is strongly just-infinite and the fact that $L$ has dense image in $G$, it follows that $K$ is actually a finite index subgroup of $G$. Hence $K$ has a finite conjugacy class in $G$, which contradicts the assumption that $S_{\varphi,G}(X)$ is continuous. So $x$ must be unique, and $\psi$ is well-defined. 

By semi-continuity again, if $(K_n)$ converges to $K$ in $S_{\varphi,G}(X)$ and $x_n = \psi(K_n)$ converges to $x$, then $\overline{\varphi(L_x^0)} \leq K$. Since $\psi(K)$ is the unique point with this property,  it follows that $x = \psi(K)$. Hence  the map $\psi$ is continuous. Note that $\psi$ is clearly $L$-equivariant.

To show that $\psi$ is almost one-to-one, denote by $X_\varphi$ the dense $G_\delta$-set of points where the map $x \mapsto \overline{\varphi(L_x^0)}$ is continuous (Proposition \ref{prop-def-S-phi-G}). Then we claim that $\psi^{-1}(\left\{x\right\}) = \left\{ \overline{\varphi(L_x^0)}\right\}$ for all $x$ in $X_\varphi$. Indeed, suppose that $K$ is such that $\overline{\varphi(L_x^0)} \leq K$. Since $K$ is in $S_{\varphi,G}(X)$, we may find a net $(x_i)$ in $X_\varphi$ such that $\overline{\varphi(L_{x_i}^0)}$ converges to $K$, and we may assume that $(x_i)$ converges to some $y$ in $X$. By semi-continuity $\overline{\varphi(L_y^0)} \leq K$, so it follows that actually $y = x$, and consequently $K = \overline{\varphi(L_x^0)}$ since $x \in X_\varphi$. This proves the claim.

\ref{item-action-extends} The Gelfand correspondence establishes a bijection between $L$-invariant $C^\ast$-sub-algebras of $C(S_{\varphi,G}(X))$ (the continuous complex valued functions on $S_{\varphi,G}(X)$) and $L$-equivariant factors of $S_{\varphi,G}(X)$. The group $G$ acts on $S_{\varphi,G}(X)$, and by density of $L$ in $G$ any $L$-invariant $C^\ast$-sub-algebra of of $C(S_{\varphi,G}(X))$ must be $G$-invariant. Hence by duality, the $L$-action on $X$ extends to a continuous $G$-action, and the map $\psi: S_{\varphi,G}(X) \to X$ is $G$-equivariant.

Finally \ref{item-action-extends-faith} easily follows from the fact that $\psi$ is almost one-to-one.
\end{proof}
	
The following result follows from the combination of Proposition \ref{prop:S-phi-G-Faithful} and Proposition \ref{prop-urs-to-X-Gamma-eq}.
	
\begin{cor} \label{cor-faith-exten}
	Let $L,X,\varphi,G$ as in the general setting. Assume that $L$ is j.n.v.n.\, and that the action of $L$ on $X$ is strongly just-infinite. Assume also that $G$ is compactly generated, monolithic, with a non-discrete, non-compact, compactly generated monolith $M$, and such that $G/M$ is compacy-by-discrete. Then the $L$-action on $X$ extends to a  continuous faithful $G$-action.
\end{cor}

%
%
%

\section{Commensurated subgroups and micro-supported actions} \label{sec-commens-micro}

\subsection{Schlichting completions}


Let $\Gamma$ be a group and $\Lambda \leq \Gamma$ a commensurated subgroup. We denote by $\Gamma/\! \!/\Lambda$ the \textbf{Schlichting completion} of the pair $(\Gamma, \Lambda)$, defined as the closure of the natural image of $\Gamma$ in the symmetric group $\Sym(\Gamma/\Lambda)$ endowed with the topology of pointwise convergence. The Schlichting completion is a tdlc group. Notice moreover that if $\Gamma$ is countable, then $\Sym(\Gamma/\Lambda)$ is a Polish group, so that $\Gamma/\! \!/\Lambda$ is Polish as well. Thus the Schlichting completion of any countable group is a second countable tdlc group. For more information on this construction we refer to \cite{ReidWesolek_Forum19, Schlichting, Shalom-Willis}. 

Clearly, every normal subgroup of $\Gamma$ is commensurated. More generally, every subgroup that is commensurate to a normal subgroup is commensurated. Such subgroups are considered as \enquote{trivial examples} of commensurated subgroups. They are characterized as follows. 

\begin{lem}\label{lem:Schlichting}
Let $\Gamma$ be a group and $\Lambda \leq \Gamma$ a commensurated subgroup. Then $\Lambda$ is commensurate to a normal subgroup of $\Gamma$ if and only if the  Schlichting completion $\Gamma/\! \!/\Lambda$  is compact-by-discrete.
\end{lem}

\begin{proof}
If  $\Gamma/\! \!/\Lambda$  is compact-by-discrete, i.e.\ if $\Gamma/\! \!/\Lambda$ has a compact open normal subgroup, then the preimage of that subgroup in $\Gamma$ is a normal subgroup that is commensurate with $\Lambda$. 

Conversely, let $N$ be a normal subgroup of $\Gamma$ that is commensurate with $\Lambda$. Denote by $\varphi \colon \Gamma \to \Gamma/\! \!/\Lambda$ the canonical homomorphism, whose image is dense. Then the closure $\overline{\varphi(N)}$    is a closed normal subgroup of $\Gamma/\! \!/\Lambda$ that is commensurate with $\overline{\varphi(\Lambda)}$. The latter is a  compact  open subgroup, so that $\overline{\varphi(N)}$  is compact and open as well. Thus $\Gamma/\! \!/\Lambda$ is compact-by-discrete. 
\end{proof}

\begin{lem}\label{lem:Schliching2}
Let $\Gamma$ be a finitely generated j.n.v.n.\ group. Let $\Lambda \leq \Gamma$ be a commensurated subgroup which is not commensurate to a normal subgroup. Then the homomorphism $\Gamma \to \Gamma/\! \!/\Lambda$ is injective. 

If in addition $\Lambda$ is not virtually contained in a normal subgroup of infinite  index in $\Gamma$, then $H = \Gamma/\! \!/\Lambda$ admits a non-discrete, compactly generated, just-non-compact quotient group $G$, and the natural  homomorphism $\varphi \colon \Gamma \to G $ is injective with dense image. 
\end{lem}
\begin{proof}
By Lemma~\ref{lem:Schlichting}, the Schlichting completion $H = \Gamma/\! \!/\Lambda$ is not compact-by-discrete.  Since $\Gamma$ is finitely generated, $H$ is compactly generated. If the canonical homomorphism $\Gamma \to H$ were not injective, then $H$ would be virtually nilpotent, hence compact-by-discrete in view of Proposition~\ref{prop:VirtuallyNilp}.

Since $H$ is compactly generated, it has a closed normal subgroup $N$ such that the quotient $G = H/N$ is just-non-compact (see \cite[Proposition~5.2]{CaMo-decomp}). We denote by $\varphi$ the composite homomorphism $\Gamma \to H \to G$. By construction $\varphi$ has dense image.

Assume now that $\Lambda$ is not virtually contained in a normal subgroup of infinite  index in $\Gamma$. 
We claim that $G$ is non-discrete. Indeed, otherwise $N$ would be open in $H$, and hence its  pre-image in $\Gamma$  would be a normal subgroup of infinite index that contains a finite index subgroup of $\Lambda$. 

Finally, we observe that the map $\varphi$ is injective, because otherwise $\varphi(\Gamma)$ would be virtually nilpotent, and hence $G$ would be virtually nilpotent as well. In view of Proposition~\ref{prop:VirtuallyNilp}, this implies that $G$ is discrete, which is not the case. 
\end{proof}

\begin{prop}\label{prop:Schliching3}
Let $\Gamma$ be a j.n.v.n.\ group such that every normal subgroup of $\Gamma$ is finitely generated. Let $\Lambda \leq \Gamma$ be a commensurated subgroup which is not commensurate to a normal subgroup of $\Gamma$.

Then  $H = \Gamma/\! \!/\Lambda$ has a closed normal subgroup such that the quotient group $G$ satisfies all the conclusions of Theorem~\ref{thm:MonolithicQuotient}. Moreover the natural homomorphism $\varphi \colon \Gamma \to G $ is injective with dense image. 
\end{prop}

\begin{proof}
By Lemma~\ref{lem:Schlichting}, the Schlichting completion $H = \Gamma/\! \!/\Lambda$ is not compact-by-discrete. Since every normal subgroup of $\Gamma$ is finitely generated, it follows that every open normal subgroup of $H$ is compactly generated. In particular $\Res(H)$ is compactly generated, by Corollary~\ref{cor:SmallOpenNorm}. Therefore, we may invoke Theorem~\ref{thm:MonolithicQuotient} to build a monolithic quotient $G$ of $H$ with the required properties. The composite map $\varphi \colon \Gamma \to H \to G$ has dense image. If $\varphi$ were not injective, then $G$ would be virtually nilpotent, hence compact-by-discrete, which violates the conclusions of Theorem~\ref{thm:MonolithicQuotient}.
\end{proof}

\subsection{The main result}

Theorem \ref{thm:Main-intro} from the introduction is a consequence of the following more comprehensive statement.

\begin{thm}\label{thm:Main}
Let $\Gamma$ be a finitely generated j.n.v.n.\ group, and let $X$ be a compact $\Gamma$-space such that the action of $\Gamma$  on $X$ is faithful, minimal and micro-supported. Assume  moreover that at least one of the following conditions holds. 
\begin{enumerate}[label=(\arabic*)]
\item \label{item-MainDyn1} $\Gamma$ has a commensurated subgroup $\Lambda$ that is of infinite index and that is not virtually contained in a normal subgroup of infinite index of $\Gamma$.

\item \label{item-MainDyn2} Every normal subgroup of $\Gamma$ is finitely generated, and  $\Gamma$ has a commensurated subgroup $\Lambda$ that is not commensurate to a normal subgroup. 

\end{enumerate}
Then the following assertions hold. 
\begin{enumerate}[label=(\roman*)]
	
\item \label{it:Main1} The action of $\Gamma$ on $X$ is an almost boundary and has compressible open subsets. Furthermore $\Gamma$ is monolithic, hence not residually finite. 

\item  \label{it:Main2}  There exists a compact $\Gamma$-space $Y$ with $Y \sim_{hp} X$, such that the $\Gamma$-action on $Y$ extends to a continuous $H$-action, where $H = \Gamma/\! \!/\Lambda$, and the quotient group $G= H/K$ of $H$ by the kernel of that action is monolithic with a non-discrete, non-amenable, compactly generated monolith $M$. Furthermore $M$ coincides with $\Res(G)$. In particular $G/M$ is compact-by-discrete.  

\item  \label{it:Main3}  
If in addition the $\Gamma$-action on $X$ is strongly just-infinite, then one can take $Y=X$ in \ref{it:Main2}.
\end{enumerate}

\end{thm}


\begin{proof}
We form the Schlichting completion $H = \Gamma/\! \!/\Lambda$, and we invoke  Lemma~\ref{lem:Schliching2} or Proposition~\ref{prop:Schliching3} depending on whether we are in situation \ref{item-MainDyn1} or \ref{item-MainDyn2}. In either case, Theorem~\ref{thm:MonolithicQuotient} can be applied (in the just-non-compact case, see also Proposition~\ref{prop:UpperStr-monolith}), we find a quotient $G$ of $H$ and an embedding $\varphi \colon \Gamma \to G$ with dense image and such that $G$ is monolithic with a compactly generated,  non-compact, non-discrete monolith $M$, and $G/M$ is compact-by-discrete.

Consider the space $Y := S_{\varphi,G}(X)$ constructed in Section \ref{subsec-URS}. The  group $G$ satisfies all the assumptions of Proposition \ref{prop:S-phi-G-Faithful}, so it follows that the $G$-action on $Y$ is faithful. Now consider the $\Gamma$-action on $E_{\varphi}(X)$. The extension $E_{\varphi}(X) \to X$ is almost one-to-one by Proposition \ref{prop-def-S-phi-G}, and hence it is highly proximal. Since the $\Gamma$-action on $X$ is micro-supported, Proposition \ref{prop-invariants-hpeq} implies that the $\Gamma$-action on $E_{\varphi}(X)$ is micro-supported. Since the $\Gamma$-action on $Y$ is also faithful (as the $G$-action is), it follows that the $\Gamma$-action on $Y$ is micro-supported by Lemma \ref{lem-MS-factor-faith}. Hence the $G$-action on $Y$ is   micro-supported. Since in addition $G$ is monolithic with a compactly generated,  non-compact and non-discrete monolith, we have shown that all the assumptions of Theorem \ref{thm:Stone-dynamics} are satisfied. The latter therefore implies that the action of $G$ on $Y$ is an almost boundary and has compressible clopen subsets. Since $\Gamma$ is dense in $G$, it follows that the action of $\Gamma$ on $Y$ is also an almost boundary with compressible clopen subsets. Now the $\Gamma$-spaces $X$ and $Y$ are highly proximally equivalent by Theorem \ref{thm-MS-all-eq}, and by Lemma~\ref{lem-highly-prox-preserve} and Proposition \ref{prop-invariants-hpeq} the property of being an almost boundary with compressible open subsets is invariant under highly proximal equivalence. Hence the space $X$ also has this property. That the group $\Gamma$ is monolithic then follows from the double commutator lemma, see Proposition I in \cite{CRW-part2}. So we have shown that \ref{it:Main1} holds.


Since the action of $G$ on $Y$ is faithful and micro-supported, by the double commutator lemma again we see that for every non-trivial normal subgroup $N$, there exists a non-empty clopen subset $U \subseteq Y$ such that $N \geq \rist_G(U)'$. In particular $N$ has a non-trivial intersection with the image of $\Gamma$. Therefore $G/N$ is virtually nilpotent. In particular, if $N$ is closed, then $G/N$ is residually discrete by Proposition~\ref{prop:VirtuallyNilp}, so that $\Res(G) \leq N$. Since $G$ is not residually discrete, we infer that  the monolith $M$ of $G$ coincides with $\Res(G)$. In particular $G/M$ is compact-by-discrete in view of Proposition~\ref{prop:Res(G)-cptly-gen}. That $M$ is not amenable follows from Theorem \ref{thm:Stone-dynamics}. So  \ref{it:Main2} is proved.


Finally if in addition the $\Gamma$-action on $X$ is strongly just-infinite, then by Proposition~\ref{prop-urs-to-X-Gamma-eq} the action of $\Gamma$ on $X$ extends to a $G$-action. Moreover Proposition~\ref{prop-urs-to-X-Gamma-eq}\ref{item-action-extends-faith} ensures that the $G$-action on $X$ is fatifhul since the  $G$-action on $Y$ is faithful by \ref{it:Main2}. This shows \ref{it:Main3}.
\end{proof}

%

Theorem~\ref{thm:Main} has at least two possible directions of applications. The first one is when the action of $\Gamma$ on $X$ is known not to satisfy conclusion \ref{it:Main1}. Then by the theorem we deduce that $\Gamma$ does not admit any commensurated subgroup as in \ref{item-MainDyn1} or \ref{item-MainDyn2}. For example we deduce the following result. See \S \ref{subsec-full-gp} and \S \ref{subsec-branch} for other illustrations.

\begin{cor}
Let $\Gamma$ be a finitely generated j.n.v.n.\ group of intermediate growth, and suppose that there exists a compact $\Gamma$-space $X$ such that the action of $\Gamma$  on $X$ is faithful, minimal and micro-supported. Then every commensurated subgroup of $\Gamma$ is commensurate to a normal subgroup. 
\end{cor}

\begin{proof}
By \cite{Rosset} every j.n.v.n.\ group $\Gamma$ of intermediate growth has the property that all normal subgroups of $\Gamma$ are finitely generated. So the statement follows from the previous theorem.
\end{proof}

The second direction of application of Theorem~\ref{thm:Main} is when the group $\Gamma$ does admit commensurated subgroups as in the assumptions. Then the conclusion \ref{it:Main3} of the theorem tells us that the action of $\Gamma$ on $X$ extends to the associated Schlichting completions (via a quotient satisfying additional properties). Examples where such a situation happens are given in \S \ref{subsubsec-higman}.

\begin{rmq} \label{rmq-not-all-extend}
	In the setting of the above theorem, it is not true that all micro-supported actions of $\Gamma$ will extend to the Schlichting completion; see Remark \ref{rmq-not-all-extend-V}. This illustrates that the assumption in \ref{it:Main3} requiring that the action is strongly just-infinite cannot be removed. 
\end{rmq}

\section{Applications to discrete groups} \label{sec-applic}

In this section we apply the results of Sections \ref{sec-embed-micro} and \ref{sec-commens-micro} to several classes of groups admitting a micro-supported action.

\subsection{Topological full groups} \label{subsec-full-gp}

\subsubsection{Preliminaries}

If $\Lambda$ is a group acting on a compact space $X$, we denote by  $\Full(\Lambda,X)$  the associated topological full group. Recall that $\Full(\Lambda,X)$ is the group of homeomorphisms $g$ of $X$ such that for every $x \in X$ there exist a neighbourhood $U$ of $x$ and an element $\gamma \in \Lambda$ such that $g(y) = \gamma(y)$ for every $y \in U$. 

Given a non-empty clopen subset $U$ of $X$ and elements $\gamma_1,\ldots,\gamma_n \in \Lambda$ such that $\gamma_1(U),\ldots,\gamma_n(U)$ are pairwise disjoint, there is an injective group homomorphism $\mathrm{Sym}(n) \to \Full(\Lambda,X)$, $\sigma \mapsto g_\sigma$, where $g_\sigma$ is defined by $g_\sigma(x) = \gamma_{\sigma(i)} \gamma_i^{-1} (x)$ if $x \in \gamma_i(U)$, and $g_\sigma(x) = x$ if $x \notin \bigcup_i \gamma_i(U)$. The \textbf{alternating full group} $\Al(\Lambda,X)$, introduced by Nekrashevych in \cite{Nek-simple-dyn}, is the subgroup of $\Full(\Lambda,X)$ generated by the images of the alternating groups $\mathrm{Alt}(n)$ under all such homomorphisms. 

The following is \cite[Theorem 4.1]{Nek-simple-dyn} (see also \cite{Matui06}). 

\begin{thm} \label{thm-nek-alt-simple}
Let $ \Lambda \acts X$ be a minimal action on a Cantor space $X$.Then every non-trivial subgroup of  $\Full(\Lambda,X)$ that is normalized by $\Al(\Lambda,X)$ contains $\Al(\Lambda,X)$. In particular $\Al(\Lambda,X)$ is simple and is contained in every non-trivial normal subgroup of $\Full(\Lambda,X)$.
\end{thm}

Recall that an action $ \Lambda \acts X$ on a Cantor space $X$ is \textbf{expansive} if there exist a compatible metric $d$ and $\delta > 0$ such that for every $x \neq y \in X$ there exists $\gamma \in \Lambda$ such that $d(\gamma(x),\gamma(y)) \geq \delta$. For the following, see Proposition 5.7 and Theorem 5.10 in \cite{Nek-simple-dyn}. 

\begin{thm} \label{thm-nek-alt-fg}
	Let $ \Lambda$ be a finitely generated group, and $ \Lambda \acts X$ a minimal and expansive action on a Cantor space $X$. Then the group $\Al(\Lambda,X)$ is finitely generated.
\end{thm}

\subsubsection{Absence of commensurated subgroups}

The goal of this paragraph is to prove Theorem \ref{thm-fullgroup-no-commens}. We will use the following lemma.

\begin{lem} \label{lem-alt-distinct-pts}
	Let $ \Lambda \acts X$ be a minimal action on a Cantor space $X$, and let $\Gamma = \Al(\Lambda,X)$. Then the $\Gamma$-action on $X$ is minimal. Moreover, for all  $x, y \in X$ with $x \neq y$, we have   $\left\langle  \Gamma_x^0, \Gamma_y^0 \right\rangle = \Gamma$.
\end{lem}

\begin{proof}
The minimality of the $\Gamma$-action follows from \cite[Lemma 3.2]{Nek-simple-dyn}. For the second statement, let $H = \left\langle  \Gamma_x^0, \Gamma_y^0 \right\rangle$, and let $O_x$ be the $\Gamma$-orbit of $x$ in $X$. We check that $H$ contains $\Gamma_z^0$ for every $z \in O_x$. That will imply that $H$ contains the normal subgroup generated by $\Gamma_x^0$, and hence that $H = \Gamma$ because $\Gamma$ is simple by Theorem \ref{thm-nek-alt-simple}.

	So given $z \in O_x$, we wish to show that $H$ contains $\Gamma_z^0$. Clearly we may assume that $z$ is distinct from $x$ and $y$, and it is enough to show that there exists $g \in H$ such that $g(x) = z$, because then $\Gamma_z^0 = g \Gamma_x^0 g^{-1} \leq H$. If $z'$ is a point in $O_x$ different from $x,y,z$, it is possible to find three disjoint clopen subsets $U_x,U_{z},U_{z'}$ containing respectively $x,z,z'$ and not containing $y$, and elements $\gamma_1,\gamma_2,\gamma_3 \in \Lambda$ such that $\gamma_1(U_x) = U_z$ and $\gamma_1(x) = z$, $\gamma_2(U_z) = U_{z'}$ and $\gamma_3(U_{z'}) = U_x$. The homeomorphism $g$ of $X$ that coincides on $U_x,U_{z}$ and $U_{z'}$ respectively with $\gamma_1,\gamma_2$ and $\gamma_3$ and which acts trivially outside $U_x \cup U_{z} \cup U_{z'}$ is an element of $\Gamma$ by definition of the alternating full group, and by construction $g$ acts trivially on a neighbourhood of $y$. Hence $g$ belongs to $H$, and $g(x) = z$, so the statement is proved.
\end{proof}

\begin{thm} \label{thm-fullgroup-no-commens}
	Let $\Lambda$ be a finitely generated group, and $ \Lambda \acts X$ a minimal and expansive action on a Cantor space $X$ such that $\Al(\Lambda,X) \acts X$ does not admit any compressible open subset. Let $\Gamma$ be a subgroup of $\Full(\Lambda,X)$ that contains $\Al(\Lambda,X)$. Then every commensurated subgroup of $\Gamma$ is either finite or contains $\Al(\Lambda,X)$.
\end{thm}

\begin{proof}
	We first treat the case $\Gamma = \Al(\Lambda,X)$. Recall that under the present assumptions, $\Gamma$ is a finitely generated simple group (Theorems \ref{thm-nek-alt-simple} and \ref{thm-nek-alt-fg}), and the $\Gamma$-action  on $X$ is minimal and strongly just-infinite by Lemma \ref{lem-alt-distinct-pts}. Hence all the assumptions of Theorem \ref{thm:Main} are verified. Hence according to part  \ref{it:Main1} of this theorem, if $\Gamma$ has a commensurated subgroup that is of infinite index and not virtually contained in a normal subgroup of infinite index, then $\Gamma \acts X$ admits a compressible open subset. By our assumption this is not the case. So $\Gamma$ has no commensurated subgroup as above, and since $\Gamma$ is simple this is equivalent to saying that every commensurated subgroup of $\Gamma$ is either finite or equal to $\Gamma$.

	We now consider an arbitrary subgroup $\Gamma$ of $\Full(\Lambda,X)$ that contains $\Al(\Lambda,X)$. Suppose that $\Sigma$ is a commensurated subgroup of $\Gamma$ that does not contain $\Al(\Lambda,X)$. Then $\Sigma ' = \Sigma \cap \Al(\Lambda,X)$ is a proper commensurated subgroup of $\Al(\Lambda,X)$, and hence is finite according to the previous paragraph. Therefore since $\Al(\Lambda,X)$ is normal in $\Gamma$, we infer that every element of $\Al(\Lambda,X)$ centralizes a finite index subgroup of $\Sigma$. Since $\Al(\Lambda,X)$ is a finitely generated group, it follows that the entire $\Al(\Lambda,X)$ centralizes a finite index subgroup of $\Sigma$. Since $\Al(\Lambda,X)$ has trivial centralizer in $\Full(\Lambda,X)$ by Theorem \ref{thm-nek-alt-simple}, the subgroup $\Sigma$ must be finite, as desired.
\end{proof}

Note that Theorem \ref{thm-fullgroup-intro} and Corollary \ref{cor-fullgroup-intro} follow from Theorem \ref{thm-fullgroup-no-commens}. In the following example we illustrate that these results fail if we remove one of the assumptions.

\begin{example} \label{ex-odom}
Let $\varphi$ denote the odometer on $X := \mathbb{Z}_p$, i.e.\ the homeomorphism of $\mathbb{Z}_p$ defined by $x \mapsto x+1$, and write $\Lambda = \left\langle \varphi \right\rangle$. The topological full group $\Full(\Lambda,X)$ is considered in details in \cite[Example 4.6]{Grig-Med}. It is an infinite group such that every finitely generated subgroup is virtually abelian.

The action of $\Lambda$ on $X$ is minimal but not expansive. If $T_p$ is the $p$-regular rooted tree naturally associated to $\mathbb{Z}_p$ (the coset tree associated to the sequence of subgroups $(p \mathbb{Z}_p, p^2 \mathbb{Z}_p, \ldots)$), then $\mathrm{Aut}(T_p)$ is naturally a subgroup of $\mathrm{Homeo}(X)$. The subgroup $\Sigma := \mathrm{Aut}(T_p) \cap \Full(\Lambda,X)$ of $\Full(\Lambda,X)$ is easily seen to be infinite, of infinite index, and commensurated in $\Full(\Lambda,X)$; and $\Sigma \cap \Al(\Lambda,X)$ is also infinite, of infinite index, and commensurated in $\Al(\Lambda,X)$. Hence this example shows that the expansivity of the action of $ \Lambda$ on $X$ in Theorem \ref{thm-fullgroup-intro} and Corollary \ref{cor-fullgroup-intro} is an essential assumption. 

Now remark that the action of $\Full(\Lambda,X)$ is minimal and expansive, but the group $\Full(\Lambda,X)$ is not finitely generated. Since $\Full(\Lambda,X)$ is equal to its own topological full group, the previous paragraph also shows that Theorem \ref{thm-fullgroup-intro} and Corollary \ref{cor-fullgroup-intro} also fail without the finite generation assumption on the original acting group.
\end{example}

\subsubsection{Extension of the action to the ambient group}

This paragraph deals with the situation where there exist non-discrete tdlc groups into which the group $\Al(\Lambda,X)$ embeds densely. Examples of such groups are the Higman-Thompson groups, or topological full group associated with a one-sided shift of finite type; see Section \ref{subsubsec-higman}.  

\begin{thm} \label{thm-full-tdlc-extends}
Let $\Lambda$ be a finitely generated group, and $ \Lambda \acts X$ a minimal and expansive action on a Cantor space $X$. Assume that $G$ is a tdlc group into which $\Al(\Lambda,X)$ embeds as a dense subgroup. Then the action of $\Al(\Lambda,X)$ on $X$ extends to an action of $G$ on $X$. 
\end{thm}

\begin{proof}
Clearly we may assume that $G$ is not discrete. The group $\Al(\Lambda,X)$ is finitely generated by Theorem \ref{thm-nek-alt-fg}, so $G$ is compactly generated. Moreover $\Al(\Lambda,X)$ is simple by Theorem \ref{thm-nek-alt-simple}, hence $G$ does not have any non-trivial discrete quotient. Therefore, by Proposition \ref{prop:JNC-quotients}, the group $G$ has a topologically simple quotient $Q$, and $\Al(\Lambda,X)$ embeds densely in $Q$. Since the action of $\Al(\Lambda,X)$ on $X$ is strongly just-infinite by Lemma \ref{lem-alt-distinct-pts}, we can apply Corollary \ref{cor-faith-exten}, which says that the action of $\Al(\Lambda,X)$ extends to the group $Q$. Since $Q$ is a quotient of $G$, in particular we have shown that the action extends to an action of $G$.
\end{proof}

\subsubsection{Examples of dense embeddings of topological full groups} \label{subsubsec-higman}

We denote by $V_{d,k}$ the Higman-Thompson group with parameters  $d \geq 2$ and $k \geq 1$ acting on the Cantor space $X_{d,k} = \{1,\ldots, k \} \cdot  \{1,\ldots, d \}^\mathbb{N}$ by prefix replacement. We refer to \cite{Hig-fp} for a precise definition of these groups. Alternatively, $V_{d,k}$ can be defined as the topological full group of a certain one-sided shift of finite type \cite{Matui-fullgp-shift}. In particular $V_{d,k}$ is equal to its own topological full group. 

The group $V_{d,k}$ admits an embedding with dense image in a non-discrete tdlc group, namely the group $\mathrm{AAut}(T_{d,k})$ of almost-automorphisms of the quasi-regular rooted tree $T_{d,k}$ \cite{Cap-deMedts}. In particular $V_{d,k}$ admits an infinite and infinite index commensurated subgroup. This fact has been recently generalized by Waltraud Lederle, who showed that more generally every topological full group associated with a one-sided shift of finite type admits infinitely many  pairwise non-commensurate commensurated subgroups \cite{Lederle-complet}. 

The group $V_{d,k}$ is well-known to be finitely generated and simple (see \cite{Hig-fp}), and it is easy to see that its action on the Cantor $X_{d,k}$ is strongly just-infinite. Hence it follows that $V_{d,k}$ satisfies all the assumptions of Theorem \ref{thm:Main}\ref{it:Main3}, so that for every dense embedding of $V_{d,k}$ into a tdlc group $G$, the action of $V_{d,k}$ on $X_{d,k}$ extends to $G$.

\begin{rmq} \label{rmq-not-all-extend-V}
In the above situation, it is not true that every minimal and micro-supported action of $V_{d,k}$ extends to $G$, and hence the assumption in Theorem~\ref{thm:Main}\ref{it:Main3} that the action is strongly just infinite is necessary.  For example in the case of $G = \mathrm{AAut}(T_{d,2})$ mentioned above (into which the group $V_{d,2}$ embeds densely), the group $G$ admits a unique minimal and micro-supported action. This follows from the fact that $G$ and the group $\mathrm{Aut}(T_{d+1})$ of automorphisms of the non-rooted regular tree of degree $d+1$ have isomorphic open subgroups (and hence have the same centralizer lattice) together with the fact that $\mathrm{Aut}(T_{d+1})$ has a unique minimal and micro-supported action \cite[Theorem B.2]{CRW-part2}. So it follows that the action of $V_{d,2}$ on $X_{d,2}$ is actually the unique minimal and micro-supported action that extends to $G$.
\end{rmq}

\subsection{Branch groups} \label{subsec-branch}


Let $T$ be a locally finite rooted tree, and $\Gamma$ a group of automorphisms of $T$. We denote by $\partial T$ the boundary of $T$. Given a vertex $v\in T$, we denote by $T_v$ the subtree of $T$ that is below $v$, and by $\partial T_v$ the associated clopen subset in $\partial T$. Note that the $\Gamma$-action on $T$ extends to an action by homeomorphisms on $\partial T$. We will denote by $\rist_\Gamma(v) := \rist_\Gamma(\partial T_v)$ the rigid stabilizer of $v$. For $n \geq 1$, we also denote by $\rist_\Gamma(n)$ the rigid stabilizer of level $n$, that is the subgroup generated by $\rist_\Gamma(v)$ when $v$ ranges over vertices of level $n$. Recall that the action of $\Gamma$ is branch if $\Gamma$ acts transitively on each level of the tree (or, equivalently, if $\Gamma$ acts minimally on $\partial T$) and if $\rist_\Gamma(n)$ has finite index in $\Gamma$ for all $n \geq 1$.

We will invoke the following result of Grigorchuk \cite{Grig-ji-branch}.

\begin{thm} \label{thm-branch-normal}
Let $\Gamma \leq \aut(T)$ be a branch group, and $N$ a non-trivial normal subgroup of $\Gamma$. Then there exists a level $n \geq 1$ such that $\rist_\Gamma(n)' \leq N$. In particular every proper quotient of $\Gamma$ is virtually abelian.
\end{thm}

We now give the proof of Theorem \ref{thm-branch-comm-intro} by applying Theorem~\ref{thm:Main}.

\begin{proof}[Proof of Theorem \ref{thm-branch-comm-intro}]
The group $\Gamma$ is j.n.v.n.\ by Theorem \ref{thm-branch-normal}. If $T$ is a rooted tree on which $\Gamma$ has a faithful branch action, then the action of $\Gamma$ on $X = \partial T$ is faithful, minimal and micro-supported. Since in addition every normal subgroup of $\Gamma$ is finitely generated by Corollary \ref{cor-norm-branch} of the appendix, it follows from part~\ref{item-MainDyn2} of Theorem~\ref{thm:Main} that if $\Gamma$ had a commensurated subgroup that is not commensurate to a normal subgroup, then conclusion \ref{it:Main1} of the theorem would hold; which is clearly not the case here. So every commensurated subgroup of $\Gamma$ is commensurate to a normal subgroup, and the statement is proved.
\end{proof}

 If $\Gamma$ is a finitely generated just-infinite branch group, Theorem \ref{thm-branch-comm-intro} says that every commensurated subgroup of $\Gamma$ is finite or of finite index. In this special case this result has been proved by different means by Wesolek in \cite{Wes-branch}. Hence Theorem \ref{thm-branch-comm-intro} extends Wesolek's result to arbitrary finitely generated branch groups.

\subsection{Groups acting on the circle} \label{subsec-circle}

We denote by $\homeo(\mathbf{S}^1)$ the homeomorphism group of the circle $\mathbf{S}^1$, and by  $\homeo^+(\mathbf{S}^1)$ the subgroup of index two consisting of orientation preserving homeomorphisms. Given a subgroup $ \Gamma \leq \homeo(\mathbf{S}^1)$, we denote by $\aut_{\Gamma}(\mathbf{S}^1)$ the centralizer of $\Gamma$ in $\homeo^+(\mathbf{S}^1)$. Recall the following well-known classification: either $\Gamma$ has a finite orbit in $\mathbf{S}^1$; or $\Gamma$ admits an exceptional minimal set (i.e.\ there exists a unique closed non-empty minimal $\Gamma$-invariant subset $K\subset \mathbf{S}^1$, and $K$ is homeomorphic to a Cantor set); or $ \Gamma$ acts minimally on $\mathbf{S}^1$. See proposition 5.6 in \cite{Ghys-circ}. Moreover for minimal actions we have the following result (see \S 5.2 in \cite{Ghys-circ}).

\begin{thm} \label{thm-Margulis}
	Assume that $\Gamma \le \homeo^+ (\mathbf{S}^1)$ acts minimally on $\mathbf{S}^1$. Then either $\aut_{\Gamma}(\mathbf{S}^1)$ is infinite and  $\Gamma$ is conjugate to a group of rotations, or $\aut_{\Gamma}(\mathbf{S}^1)$ is a finite cyclic group, and the action of $\Gamma$ on the topological circle  $\aut_{\Gamma}(\mathbf{S}^1) \backslash \mathbf{S}^1$ is proximal.
\end{thm}

We will use the following lemma.

\begin{lem} \label{lem-circ-commens-min}
Let $ \Gamma$ be a subgroup of $\homeo^+(\mathbf{S}^1)$ that is minimal and not conjugated to a group of rotations, and let $\Lambda$ be an infinite commensurated subgroup of $\Gamma$. Then the action of $\Lambda$ on $\mathbf{S}^1$ is minimal.
\end{lem}

\begin{proof}
We have to show that $\Lambda$ cannot have a finite orbit or an exceptional minimal set. Upon passing to the action on $\aut_{\Gamma}(\mathbf{S}^1) \backslash \mathbf{S}^1$, by Theorem \ref{thm-Margulis} we may assume that the action of $\Gamma$ on $\mathbf{S}^1$ is proximal. Note that here this is equivalent to saying that every proper closed subset is compressible.

 Suppose that $\Lambda$ has a finite orbit in $\mathbf{S}^1$. Then upon passing to a finite index subgroup, we may assume that the set $F$ of $\Lambda$-fixed points is non-empty. Note that $F$ is not the entire circle since $\Lambda$ is infinite. Consider the action of $\Lambda$ on a $\Lambda$-invariant open interval in the complement of $F$. By considering powers of a suitable element, we see that there exists a non-empty open interval $I$ such that every finite index subgroup of $\Lambda$ contains an element $\lambda$ such that $\lambda(I)$ and $I$ are disjoint. Fix such an interval $I$. By minimality and proximality of the $\Gamma$-action, there is $\gamma \in \Gamma$ such that $\gamma(F) \subset I$. It follows that $\Lambda \cap \gamma \Lambda \gamma^{-1}$ is a finite index subgroup  of $\Lambda$ that fixes $\gamma(F) \subset I$, and we obtain a contradiction with the definition of $I$.

Suppose now that $\Lambda$ has an exceptional minimal set $K$. Again by minimality and proximality, we can find $\gamma \in \Gamma$ such that $\gamma(K)$ lies in a connected component $J$ of the complement of $K$ in $\mathbf{S}^1$. Then the subgroup $\Lambda \cap \gamma \Lambda \gamma^{-1}$ stabilizes $K$ and $\gamma(K)$. Hence $\Lambda \cap \gamma \Lambda \gamma^{-1}$ stabilizes $J$, and therefore cannot be of finite index in $\Gamma$. Contradiction.
\end{proof}

Whenever $\Gamma$ is a subgroup of $\homeo(\mathbf{S}^1)$, we denote by $\Gamma^0$ the subgroup of $\Gamma$ generated by the $\Gamma_x^0$, $x \in \mathbf{S}^1$. Equivalently, $\Gamma^0$ is the subgroup of $\Gamma$ generated by the elements that fix pointwise an open interval in $\mathbf{S}^1$.

\begin{prop} \label{prop-mon-circle}
Let $ \Gamma$ be a subgroup of $\homeo^+(\mathbf{S}^1)$ that is minimal and micro-supported. Then $\Gamma$ is monolithic, with monolith $M = [\Gamma^0,\Gamma^0]$, and $M$ is simple.
\end{prop}

\begin{proof}
Since the action of $\Gamma$ on the circle is micro-supported, the group $\aut_{\Gamma}(\mathbf{S}^1)$ must be trivial, and hence the action of $\Gamma$ is proximal by Theorem \ref{thm-Margulis}. Since here proximality is equivalent to the fact that every proper closed subset is compressible, the statement then follows from a general result about such actions, see Proposition 4.6 in \cite{LB-ame-urs-latt}.
\end{proof}

We are now ready to prove Theorem \ref{thm-circle-intro} from the introduction.

\begin{proof}[Proof of  Theorem \ref{thm-circle-intro}]
First observe that since the action of $\Gamma$ on the circle is micro-supported, the group $\aut_{\Gamma}(\mathbf{S}^1)$ must be trivial, and hence the action of $\Gamma$ is proximal by Theorem \ref{thm-Margulis}.
	
Suppose that $\varphi: \Gamma \to G$ is a dense embedding of $\Gamma$ into a tdlc group $G$. We aim to show that $G$ is discrete. Upon passing to a subgroup of index at most two, we may assume that $\Gamma$ acts on $\mathbf{S}^1$ by preserving the orientation. We argue by contradiction and assume that $G$ is not discrete. Let $U$ be a compact open subgroup of $G$, and $\Lambda = \varphi^{-1}(U)$. The subgroup $U$ is an infinite commensurated subgroup of $\Gamma$, and hence acts minimally on the circle by Lemma \ref{lem-circ-commens-min}. Consider the space $E_\varphi(\mathbf{S}^1)$ and $S_{\varphi,G}(\mathbf{S}^1)$ constructed in Section \ref{subsec-URS}. Since the extension $E_\varphi(\mathbf{S}^1) \to \mathbf{S}^1$ is almost one-to-one by Proposition \ref{prop-def-S-phi-G}, $\Lambda$ also acts minimally on $E_\varphi(\mathbf{S}^1)$, and hence also on $S_{\varphi,G}(\mathbf{S}^1)$. Therefore $S_{\varphi,G}(\mathbf{S}^1)$ is a compact $G$-space on which every compact open subgroup of $G$ acts minimally. This is possible only if $S_{\varphi,G}(\mathbf{S}^1)$ is a point, so $S_{\varphi,G}(\mathbf{S}^1) = \left\lbrace N\right\rbrace $ for some closed normal subgroup $N$ of $G$. By definition of $S_{\varphi,G}(\mathbf{S}^1)$, it follows that there is a dense subset $X_\varphi \subseteq \mathbf{S}^1$ such that $\overline{\varphi(\Gamma_x^0)} = N$ for every $x \in X_\varphi$. Now if $z \in \mathbf{S}^1$ is arbitrary and $\gamma$ is an element of $\Gamma_z^0$, then by density we can find  $x \in X_\varphi$ such that $\gamma \in \Gamma_z^0$. Hence it follows that $N$ contains $\varphi(\Gamma^0)$, and hence $N = \overline{\varphi(\Gamma^0)}$. Since $\Gamma^0$ has finite index in $\Gamma$ by assumption, it follows that $N$ has finite index in $G$, and that $N$ is compactly generated. According to Lemma \ref{lem:Key} we can find an open interval $I$ such that $\varphi(\rist_\Gamma(I))$ lies in $B(G)$. By proximality of the $\Gamma$-action and since $B(G)$ is a normal subgroup of $G$, it follows that $\varphi(\rist_\Gamma(J))$ lies in $B(G)$ for every open interval $J$, ie $\varphi(\Gamma^0) \leq B(G)$. Since in addition $B(G)$ is a closed subgroup of $G$ by Theorem \ref{thm-b(g)-closed} since $G$ is compactly generated, it follows that $B(G)$ is a finite index open subgroup of $G$. So by Theorem \ref{thm-usakov}, we deduce that $G$ admits a compact open normal subgroup $K$. Consider the subgroup $\varphi^{-1}(K)$. It is a non-trivial normal subgroup of $\Gamma$, and hence contains the monolith $M$ of $\Gamma$, and $M$ is simple by Proposition \ref{prop-mon-circle}. Since no simple group can embed into a profinite group, we have obtained a contradiction.

Now let $\Lambda$ be a commensurated subgroup of $\Gamma$, and assume for a contradiction that $\Lambda$ is not commensurate to a normal subgroup of $\Gamma$. Since $\Gamma^0$ has finite index in $\Gamma$ by assumption, every proper quotient of $\Gamma$ is virtually abelian by Proposition \ref{prop-mon-circle}. Hence by Lemma \ref{lem:Schliching2} the homomorphism $\Gamma \to \Gamma/\! \!/\Lambda$ is injective. But we have seen in by the previous paragraph that this implies that $\Gamma/\! \!/\Lambda$ is discrete, so we have obtained a contradiction. So $\Lambda$ must be commensurate to a normal subgroup of $\Gamma$, and the statement is proved.
\end{proof}

 \begin{example} \label{ex-T-circle}
Fix two integers $\ell, k \geq 1$. Consider some integers $n_i \geq 2$, and write $\overline{n} = (n_1,\ldots,n_k)$. Let $P$ be the multiplicative group of $\mathbf{R}^+_{>0}$ generated by $n_1,\ldots,n_k$, and we denote by $A = \mathbf{Z}[1/n_1,\ldots,1/n_k]$ the ring of $\overline{n}$-adic rationals, ie rational numbers whose denominator is in $P$. Denote by $T(\ell,A,P)$ the group of piecewise linear homeomorphisms of $\mathbf{R} / \ell \mathbf{Z}$ with finitely many breakpoints, all in $A$, with slopes in $P$, and which preserve the $\overline{n}$-adic rationals. In case $\ell=k=1$ and $n_1=2$, the group $T(1,\mathbf{Z}[1/2],2^{\mathbf{Z}})$ is Thompson's group $T$. 

If we write $\Gamma = T(\ell,A,P)$, the quotient $\Gamma / \Gamma^0$ is finite, and was described in \cite{Stein-PL} (see Theorem 5.2 and Lemma 5.4 there). Hence Theorem \ref{thm-circle-intro} applies to this family of groups. In the case $\ell=k=1$ and $n_1=2$, we recover the fact from \cite{LB-Wes} that every proper commensurated subgroup of Thompson's group $T$ is finite.
 \end{example}

\section{Dense embeddings between tdlc groups} \label{subsec-dense-LC}

In this final section, we apply the results of this paper to the situation where $\varphi \colon H \to G$ is a continuous homomorphisms with dense image and $G, H$ are both non-discrete tdlc groups. General results on dense embeddings into such topologically simple groups have been established in \cite{CRW_dense}. We start by recalling some terminology from loc. cit. 

\subsection{Robustly monolithic groups} 

Following \cite{CRW_dense}, we say that a tdlc group $G$ is \textbf{expansive} if there exists a compact open subgroup $U \leq G$ such that $\bigcap_{g\in G} gUg^{-1} = 1$. When $G$ is compactly generated, this is equivalent to asking that $G$ admits a Cayley-Abels graph on which it acts faithfully. We say that $G$ is \textbf{regionally expansive} if $G$ contains a compactly generated open subgroup that is expansive. Notice that any such group $G$ is first countable. We will also use the following terminology from \cite{CRW_dense}.

\begin{defin}
We say that a tdlc group $G$ is \textbf{robustly monolithic} if $G$ is monolithic, and its monolith is non-discrete, regionally expansive and topologically simple. The class of robustly monolithic tdlc groups is denoted by $\mathscr R$.
\end{defin}

By definition a group in $\mathscr R$ is necessarily non-discrete. It turns out that a group in $\mathscr R$ is itself regionally expansive \cite[Proposition 5.1.2]{CRW_dense}, and hence in particular first countable.

 Notice the inclusion $\mathscr  S_{\mathrm{td}} \subset \mathscr R$, where $\mathscr  S_{\mathrm{td}}$ is the class of compactly generated tdlc groups that are non-discrete and topologically simple. The main motivation to introduce the class $ \mathscr R$ is provided by \cite[Theorem~1.1.2]{CRW_dense}, which ensures the class $\mathscr R$ is stable under taking dense locally compact subgroups. More precisely, given a continuous injective homomorphism of tdlc groups $\varphi \colon H \to G$ with dense image such that $H$ is non-discrete and $G \in \mathscr R$, then $H \in \mathscr R$. In particular, every non-discrete dense locally compact subgroup of a group in  $\mathscr  S_{\mathrm{td}}$ belongs to $\mathscr R$ (but it may fail to be in $\mathscr  S_{\mathrm{td}}$). 
 
We recall the following result  \cite[Theorem~1.2.5]{CRW_dense}.

\begin{thm}\label{thm:[A]-semisimple}
	Every group  $G \in \mathscr R$ is [A]-semisimple.
\end{thm}

The following observation will be used in the next subsections.
	
	\begin{prop} \label{prop-R-B(G)trivial}
		Every group $G \in \mathscr R$ satisfies $B(G) = 1$.
	\end{prop}
	
	\begin{proof}
		Suppose first that $G$ is compactly generated. In that case $B(G)$ is a closed subgroup of $G$ (Theorem \ref{thm-b(g)-closed}). If $B(G)$ is non-trivial, then $B(G)$ contains the monolith $M$ of $G$, and it follows from Theorem \ref{thm-usakov} that every closed compactly generated subgroup subgroup of $M$ is compact-by discrete. But by the definition of the class $\mathscr R$, we have $M \in \mathscr R$, so by  \cite[Th.~1.1.3]{CRW_dense} the group $M$ contains a compactly generated open subgroup $O$ such that $O \in \mathscr R$. Such a subgroup $O$ cannot be compact-by discrete, which is absurd. 
		
		In the general case, we invoke again \cite[Th.~1.1.3]{CRW_dense}, which provides a compactly generated open subgroup $O$ of $G$ such that $O \in \mathscr R$. By the previous paragraph we have that $B(O)$ is trivial, so $B(G)$ intersects $O$ trivially. Hence $B(G)$ is a discrete normal subgroup of $G$, which must therefore be trivial since $G$ has trivial quasi-center by Theorem \ref{thm:[A]-semisimple}.
	\end{proof}

\subsection{$\Omega_G$ as a maximal highly proximal extension}

As a consequence of \ref{thm:[A]-semisimple}, the centralizer lattice $\LC(G)$ is a Boolean algebra, and its Stone space $\Omega_G$ satisfies the universal property from Theorem~\ref{thm-LCG-MS}. That universal property implies in particular that if $X$ is a  totally disconnected compact $G$-space on which the $G$-action is faithful and micro-supported, then that $G$-action on $X$ is minimal and strongly proximal. In this section we highlight additional features of the $G$-space $\Omega_G$, showing in particular that its universal property holds among all faithful micro-supported compact $G$-spaces, not only the totally disconnected ones. We start with the following.

\begin{prop}\label{prop:micro-comm->minimal}
	Let $G  \in \mathscr R$ and $X$ be a compact $G$-space. If the  $G$-action on $X$ is faithful and micro-supported, then the $\Mon(G)$-action is micro-supported and  minimal. In particular the $G$-action is minimal. 
\end{prop}

\begin{proof}
We first claim that the action of the monolith $M = \Mon(G)$ of $G$ on $X$ is also micro-supported. Indeed, otherwise we can find a non-empty open subset $V$ of $X$ such that $M \cap \rist_G(V) = \rist_M(V)$ is trivial. Let $W \subseteq V$ be a non-empty open subset with $\overline W \subseteq V$. Then the set $\{g \in G \mid g\overline W \subseteq V\}$ is an identity neighbourhood in $G$. Therefore, it contains a compact open subgroup $U$ of $G$. Set $V' = \bigcup_{u \in U} uW$. Hence $V'$ is a $U$-invariant non-empty open subset of $X$ which is entirely contained in $V$. It follows that $\rist_G(V')$ is normalized by $U$, hence it is a locally normal subgroup of $G$. Moreover, it is non-trivial since the $G$-action is micro-supported, and the intersection  $M  \cap \rist_G(V')$ is trivial since $V' \subseteq V$ and $\rist_M(V)$ is trivial. It follows that  $M$ and $\rist_G(V')$ are two locally normal subgroups of $G$ that intersect trivially. Hence they commute by  \cite[Th. 3.19 (iii)]{CRW-part1}.  In particular $C_G(M)$ is a non-trivial normal subgroup of $G$, and hence $M \leq C_G(M)$. So $M$ is abelian, contradicting the hypothesis that $G \in \mathscr R$. This proves the claim. 

In view of the claim,  we may assume without loss of generality that $M=G$, i.e.\ $G$ is topologically simple. By \cite[Th.~1.1.3]{CRW_dense}, we can find a compactly generated open subgroup $O$ of $G$ such that $O \in \mathscr R$. As above, if the $O$-action on $X$ were not micro-supported, we could find a locally normal subgroup of $G$ of the form $\rist_G(V)$ that intersects $O$ trivially. In particular $\rist_G(V)$ would be discrete, and hence contained in $\QZ(G)$ by Lemma 3.2 in \cite{CRW-part1}. Since $\QZ(G)$ is trivial by Theorem~\ref{thm:[A]-semisimple}, we infer that $\rist_G(V)$ is also trivial, a contradiction. Hence the $O$-action on $X$ is micro-supported. 

Suppose for a contradiction that there exists $x \in X$ such that the orbit closure $Y$ of $x$ is a proper subset of $X$. Then $\rist_G(X \setminus Y)$ is a non-trivial closed normal subgroup of $G$, and hence $\rist_G(X \setminus Y) = G$.  We infer that $G$ acts trivially on $Y$, and hence $G$ fixes $x$. The subgroup $G_x^0$ is therefore a non-trivial normal subgroup of $G$, and hence $G_x^0$ is dense in $G$. Since $O$ is open, it follows that $O_x^0 = G_x^0 \cap O$ is dense in $O$.  Since $O$ is compactly generated and its action on $X$ is micro-supported, we may apply Lemma \ref{lem:Key} to the group $O$ (with $\varphi: O \to O$ the identity and $H$ the trivial subgroup). By the conclusion of this lemma we deduce in particular that $B(O)$ is non-trivial. This provides a contradiction with Proposition \ref{prop-R-B(G)trivial}, which finishes the proof.
\end{proof}

The following result refines the universal property of $\Omega_G$. The notions of maximal highly proximal extension and highly proximal equivalence have been defined in Section \ref{sec-prelim}.

\begin{thm}\label{thm:Omega_G}
Let $G \in \mathscr R$. Then all faithful, micro-supported compact $G$-spaces are minimal and highly proximally equivalent. Their common maximal highly proximal extension is $\Omega_G$.
\end{thm}

\begin{proof}
Let $X$ be a compact $G$-space on which the $G$-action is faithful and micro-supported, and let $X^*$ be the maximal highly proximal extension of $X$. We shall prove that $X^*$ is isomorphic to $\Omega_G$. 

By Proposition~\ref{prop:micro-comm->minimal}, the $G$-action on $X$ is minimal. Hence the the $G$-action on $X^*$ is minimal, and it is micro-supported by Proposition \ref{prop-invariants-hpeq}. Moreover $X^*$  is totally disconnected by Proposition \ref{prop-HPM-totdisc}. Since the $G$-action on $X^*$ is also faithful, we may apply Theorem~\ref{thm-LCG-MS}, which ensures that there exists a factor $G$-map $\pi: \Omega_G \to X^*$. By Lemma \ref{lem-ext-MS-necessarilyHP}, $\pi: \Omega_G \to X^*$ is necessarily highly proximal. Since $X^*$ is by definition the maximal highly proximal extension of $X$, the map $\pi: \Omega_G \to X^*$ must be an isomorphism.
\end{proof}

Combining the results above with  \cite[Th.~7.3.3]{CRW_dense}, we obtain the following. 

\begin{cor}\label{cor:Omega_G}
Let $G \in \mathscr R$ and $X$ be a compact $G$-space on which the $G$-action is faithful and micro-supported. Then the $\Mon(G)$-action (hence also the $G$-action) on $X$ is micro-supported, minimal and strongly proximal. Moreover it has a compressible open set. 
\end{cor}
\begin{proof}
By Proposition~\ref{prop:micro-comm->minimal}, the $\Mon(G)$-space $X$ is minimal and micro-supported. By 
 Theorem~\ref{thm:Omega_G},  there is a factor $G$-map  $\Omega_G \to X$ which is highly proximal. 
 By \cite[Th.~7.3.3]{CRW_dense}, the $\Mon(G)$-action on $\Omega_G$ is minimal, strongly proximal and has a compressible open set. The required conclusions  now follow  from  Proposition~\ref{prop-invariants-hpeq}, keeping  in mind that strong proximality is inherited by factors.
\end{proof}

\subsection{Behavior of $\Omega_G$ with respect to dense embeddings}

\begin{thm}\label{thm:Dense-embedding-simple}
Let $G, H \in \mathscr R$ and  $\varphi \colon H \to G$ be a continuous injective homomorphism with dense image, and assume that $G$ is compactly generated. Suppose in addition that $H/\Mon(H)$ is compact, or that $G/\Mon(G)$ is compact.

If $\Omega_H$ is non-trivial, then the $H$-action on $\Omega_G$ is faithful and micro-supported, and there exists a $H$-map $\Omega_H \to \Omega_G$ that is a highly proximal extension.
\end{thm}

\begin{proof}
Since $\Omega_H$ is non-trivial, it follows from \cite[Proposition~7.2.3]{CRW_dense} that the $H$-action on $\Omega_H$ is faithful. Hence by Theorem \ref{thm-LCG-MS} it is micro-supported, and by \cite[Theorem~1.2.6]{CRW_dense} it is minimal and strongly proximal. Since $G$ is second countable, the conditions from Section~\ref{sec:GeneralSetting} are thus fulfilled.

Consider the spaces $E_\varphi(\Omega_H)$ and $Y := S_{\varphi,G}(\Omega_H)$ from Section \ref{subsec-URS}. Recall that $E_\varphi(\Omega_H)$ is an almost one-to-one extension of $\Omega_H$ by  Proposition \ref{prop-def-S-phi-G}. Since $\Omega_H$ admits no non-trivial highly proximal extension by Theorem~\ref{thm:Omega_G}, it follows that the map from $E_\varphi(\Omega_H)$ to $\Omega_H$ is actually an isomorphism and that $Y$ is a factor of $\Omega_H$. 

We claim that the $H$-action on $Y$ is faithful. Suppose for a contradiction that the $H$-action on $Y$ is not faithful. Then the $G$-action on $Y$ is not faithful either and the subgroup $N = \overline{\varphi(\Mon(H))} \leq G$ acts trivially. If $H/\Mon(H)$ is compact then $G/N$ is compact since $H$ is dense in $G$, and if $G/\Mon(G)$ is compact then $G/N$ is also compact since $N$ contains $\Mon(G)$. Hence in both situations the $G$-action on $Y$ factors through a compact group. On the other hand, recall that the $H$-action on $\Omega_H$ is strongly proximal. Since strong proximality passes to factors, the $H$-action on $Y$ is strongly proximal. Since the only strongly proximal action of a compact group is the trivial action, it follows that $Y$ is a singleton. In other words $J = \overline{\varphi(H^0_x)}$ is a normal subgroup of $G$ for all $x \in \Omega_H$. 

In the case where $G/\Mon(G)$ is compact,  every non-trivial closed normal subgroup of $G$ is cocompact, so in particular compactly generated. Hence in this situation we can apply Lemma~\ref{lem:Key}, which ensures that $B(J) \leq B(G)$ is non-trivial. We derive a contradiction as the hypothesis that $G \in \mathscr R$ implies that $B(G)$ is trivial  (otherwise by the same argument as in the proof of Proposition \ref{prop-urs-phi-nontrivial-2}, we would obtain a closed normal subgroup of $G$ that is compact-by-discrete, which is impossible here). Hence it only remains to obtain a contradiction when $H/\Mon(H)$ is compact. Let $T \leq H$ be a non-trivial closed subgroup whose normalizer $N_H(T)$ is of finite index. Then $N_H(T)$ contains an open normal subgroup of $H$, and hence we have   $\Mon(H) \leq N_H(T)$. Thus $T$ and $\Mon(H)$ normalize each other. In particular $T \cap \Mon(H)$ is a closed normal subgroup $\Mon(H)$. Since the latter is topologically simple, with a trivial centralizer in $H$, it follows that $\Mon(H)$ is contained in $T$, and hence $T$ is cocompact in $H$. Hence we may invoke Proposition~\ref{prop-urs-phi-nontrivial-1}, and again since $B(G)$ is trivial we have reached a contradiction. Consequently, the $H$-action on $Y$ is faithful, and  it follows from Lemma~\ref{lem-MS-factor-faith} that the $H$-action on $Y$ is also micro-supported. In particular, the $G$-action on $Y$ is micro-supported as well. 

We claim that the $G$-action  on $Y$ is also faithful. According to the previous paragraph, we are in position to apply Proposition~\ref{prop:FaithfulDenseSubgroup}, which provides a $G$-equivariant order-preserving injective map $f \colon \mathcal A \to \LC(G)$, where  $\mathcal A$ is   the Boolean algebra of clopen subsets of $Y$. If the $G$-action on $Y$ were not faithful, then $\Mon(G)$ would act trivially on $Y$, hence also on $f(\mathcal A)$. On the other hand, by \cite[Proposition~7.2.3]{CRW_dense}, the only $\Mon(G)$-fixed points in $\LC(G)$ are the trivial ones $0$ and $\infty$, a contradiction. This proves the claim. 

We may now invoke Theorem~\ref{thm:Omega_G}, and   deduce that there is $G$-map $\Omega_G \to Y$ that is a highly proximal extension. Since the $H$-action on $Y$ is micro-supported, by Proposition~\ref{prop-invariants-hpeq} the $H$-action on $\Omega_G$ is also micro-supported, and the statement follows by applying again Theorem~\ref{thm:Omega_G}.
\end{proof}

Recall that if $K$ is a compact subgroup of a tdlc group $G$, the commensurator $\Comm_G(K)$ carries a unique tdlc group topology such that the inclusion map $K \to \Comm_G(K)$ is continuous and open (see for instance \cite[Lemma 5.13]{CRW-part2}). We denote by $G_{(K)}$ the group $\Comm_G(K)$ endowed with that topology.

Combining the above theorem with a result from \cite{CRW_dense}, we derive the following consequence. 

\begin{cor}\label{cor:Localizations}
Let $G \in \mathscr S_{\mathrm{td}}$ and $K \leq G$ be an infinite compact subgroup. Assume that one of the following conditions holds:
\begin{enumerate}[label=(\arabic*)]
\item $K$ is locally normal and $\Comm_G(K) = G$. 

\item $K$ is a pro-$p$ Sylow subgroup of some compact open subgroup of $G$, for some prime $p$. 
\end{enumerate}

Then there exists a $G_{(K)}$-map $\Omega_{G_{(K)}} \to \Omega_G$ that is a highly proximal extension.

\end{cor}

\begin{proof}
If (1) holds, we have $\Comm_G(K) = G$ as abstract groups. If (2) holds, we know from \cite[Theorem~1.2]{Reid_Sylow} that $\Comm_G(K)$ is dense in $G$. In either case by \cite[Theorem~1.1.2]{CRW_dense}, we have $G_{(K)} \in \mathscr R$. 

Observe that $\LC(K) = \LC(G_{(K)})$. Clearly there is nothing to prove if $\LC(G)$ and $\LC(G_{(K)})$ are both trivial. Assume first that $\LC(G)$ is non-trivial. If (1) holds, then the action of $G_{(K)}$ on $\Omega_G$ is faithful and micro-supported as $G_{(K)} = G$ as abstract groups. If (2) holds then the same is true according to \cite[Theorem~8.4.1]{CRW_dense}. Hence in either case the statement follows from Theorem \ref{thm:Omega_G}. Conversely if $\LC(K)$ is non-trivial, then we are in position to apply Theorem \ref{thm:Dense-embedding-simple} since $G$ lies in $\mathscr S_{\mathrm{td}}$, which gives the statement. 
\end{proof}


It follows in particular that if $G \in \mathscr S_{\mathrm{td}}$ is of \textbf{atomic type} in the terminology of \cite[Theorem~F]{CRW-part2}, then every  locally normal subgroup $K$ of $G$ has a trivial centralizer lattice. Indeed, by \cite[Theorem F]{CRW-part2}, if the group $G$ is of atomic type, then its action on the structure lattice $\LN(G)$ is trivial. This implies that every compact locally normal subgroup $K$ is commensurated in $G$, and that the centralizer lattice $\LC(G)$ is trivial (since otherwise $G$ would act non-trivially, hence faithfully, on $\LC(G)$ by \cite[Theorem J]{CRW-part2}; its action on $\LN(G)$ is thus a fortiori faithful). Hence by the above theorem $\LC(K)$ is trivial.

\appendix

\section[Normal subgroups of f.g. branch groups are f.g., by D.~Francoeur]{Normal subgroups of finitely generated branch groups are finitely generated --- by Dominik Francoeur}

Our goal is to   prove that every normal subgroup of a finitely generated branch group is finitely generated, see Corollary~\ref{cor-norm-branch}. We begin with a lemma, which is a variation of what is sometimes known as the \emph{double commutator lemma}. This result has appeared in various forms and degrees of generality over the years. It states that for every non-trivial normal subgroup $N$ of a group $G$ of homeomorphisms of a Hausdorff space $X$, there exists a non-empty open subset $U\subseteq X$ such that $\rist'_G(U)\leq N$, where $\rist'_G(U)$ is the derived subgroup of the rigid stabiliser $\rist_G(U)$. In the following lemma, we prove that under some stronger conditions on the action of $G$, we can find in $N$ not only $\rist'_G(U)$, but a finitely generated subgroup containing it. As the proof involves taking three commutators instead of the usual two, it was suggested to us that we name it the \emph{triple commutator lemma}.

\begin{lem}[Triple commutator lemma]\label{lemma:TripleCommutatorLemma}
	Let $G$ be a group with a micro-supported faithful action by homeomorphisms on a Hausdorff space $X$. Suppose that there exists a base $B$ for the topology on $X$ such that for every open subset $U\in B$, the rigid stabiliser $\rist_G(U)$ is finitely generated. Then, for every non-trivial normal subgroup $N\trianglelefteq G$, there exists a non-empty open subset $U\in B$ and a finitely generated subgroup $H\leq N$ such that $\rist'_G(U)\leq H\leq N$.
\end{lem}

\begin{proof}
	Let $g\in N$ be a non-trivial element of $N$. Since the action of $G$ on $X$ is faithful, there exists an open subset $V\subseteq X$ such that $gV\ne V$. Since $X$ is Hausdorff, replacing $V$ by a smaller open subset if necessary, we can assume that $gV\cap V = \varnothing$.
	
	As the action of $G$ on $X$ is micro-supported, there must exist an element $t\in \rist_G(gV)$ and an open subset $W\subseteq gV$ such that $tW\cap W = \varnothing$. Let us choose a non-empty basic open subset $U\in B$ contained in $g^{-1}W$. We then have that $U, gU$ and $tgU$ are three pairwise disjoint open subsets of $X$.
	
	By our hypothesis, $\rist_G(U)$ is finitely generated. Let $\{r_1,\dots, r_n\}$ be a finite symmetric generating set for $\rist_G(U)$, and let $H\leq G$ be the subgroup generated by
	\[\left\{[g,r_1^{-1}], \dots, [g, r_n^{-1}]\right\} \cup \left\{ [tgt^{-1},r_1^{-1}], \dots, [tgt^{-1}, r_n^{-1}]\right\}.\]
	The fact that $N$ is normal implies directly that $H\leq N$. We will now show that $\rist'_G(U)\leq H$.
	
	Let $r\in \rist_G(U)$ be an arbitrary element. Since $\{r_1,\dots, r_n\}$ is a symmetric generating set for $\rist_G(U)$, there exist $1\leq i_1,\dots, i_k \leq n$ such that $r=r_{i_1}\dots r_{i_k}$. For $1\leq i \leq n$, we have $[g,r_i^{-1}] = (gr_i^{-1}g^{-1})r_i$, with $r_i\in \rist_G(U)$ and $gr_i^{-1}g^{-1}\in \rist_G(gU)$. Since $U$ and $gU$ are disjoint open subsets of $X$, the elements of $\rist_G(U)$ commute with the elements of $\rist_G(gU)$. Thus, we have
	\[[g, r^{-1}]=gr^{-1}g^{-1}r=[g,r_{i_1}^{-1}]\dots [g, r_{i_k}^{-1}]\in H.\]
	Using the same notation, we also have, for $1\leq i \leq n$,
	\[[tgt^{-1}, r_i^{-1}] = (tgt^{-1}r_i^{-1}tg^{-1}t^{-1})r_i = ((tg)r_i^{-1}(tg)^{-1})r_i,\]
	where the last equality comes from the fact that $t$ and $r_i$ commute, since $t\in \rist_G(gV)$ and $r_i\in rist_G(U)\leq \rist_G(V)$, with $V\cap gV=\varnothing$. Note that we have $(tg)r_i^{-1}(tg)^{-1} \in \rist_G(tgU)$ and $r_i\in \rist_G(U)$. Since $U\cap tgU = \varnothing$, elements of $\rist_G(U)$ and $\rist_G(tgU)$ commute. Therefore, as above, we have
	\[[tg,r^{-1}]=(tg)r^{-1}(tg)^{-1}r=[tgt^{-1},r_{i_1}^{-1}]\dots [tgt^{-1}, r_{i_k}^{-1}]\in H.\]
	
	It follows that for all $r,s\in \rist_G(U)$, we have
	\[[[g,r^{-1}],[tg,s^{-1}]]=[(gr^{-1}g^{-1})r, ((tg)s^{-1}(tg)^{-1})s] \in H.\]
	However, since $gr^{-1}g^{-1} \in \rist_G(gU)$ and $(tg)s^{-1}(tg)^{-1}\in \rist_G(tgU)$, we get, using the fact that $U, gU$ and $tgU$ are pairwise disjoint sets, that $gr^{-1}g^{-1}$ commutes with $r$, $s$ and $(tg)s^{-1}(tg)^{-1}$, and that $(tg)s^{-1}(tg)^{-1}$ commutes with $r$, $s$ and $gr^{-1}g^{-1}$. It follows that
	\[[[g,r^{-1}],[tg,s^{-1}]]=[r,s],\]
	and so $[r,s]\in H$. The result follows immediately from the fact that $\rist'_G(U)$ is generated by elements of the form $[r,s]$ with $r,s\in \rist_G(U)$.
\end{proof}

Using this lemma, we can conclude that every normal subgroup of a finitely generated branch group is itself finitely generated. In fact, we will prove a slightly more general result. Let us first recall a few definitions.

\begin{defin}
	A group $G$ is said to be \textbf{Noetherian}, or to satisfy the maximum condition on subgroups, if every subgroup of $G$ is finitely generated.
\end{defin}

\begin{defin}
	Let $T$ be a spherically homogeneous locally finite rooted tree $T$. A subgroup $G\leq \text{Aut}(T)$ of the group of automorphisms of $T$ is said to be a \textbf{weakly branch group} if it acts transitively on each level of the tree and if for all vertex $v\in T$, the rigid stabiliser $\rist_G(v)$ is non-trivial.
\end{defin}

\begin{thm}\label{thm:NormalSubgroupsAreFG}
	Let $G$ be a finitely generated weakly branch group acting on a spherically homogeneous locally finite rooted tree $T$. If $\rist_G(n)$ is finitely generated and if $G/\rist'_G(n)$ is Noetherian for all $n\in \N$, then every normal subgroup  $N\trianglelefteq G$  is finitely generated.
\end{thm}
\begin{proof}
	If $N=1$, then it is obviously finitely generated. Let us now assume that $N\ne 1$.
	
	Since $G$ is a weakly branch group, its action on the boundary of the rooted tree $T$, which is homeomorphic to the Cantor set, is micro-supported. If we assume that $\rist_G(n)$ is finitely generated for all $n\in \N$, it then satisfies the hypotheses of Lemma \ref{lemma:TripleCommutatorLemma}. Therefore, there must exist a vertex $v\in T$ on level $n\in \N$ and a finitely generated subgroup $H\leq N$ such that $\rist'_G(v) \leq H$. Let $\{t_1, \dots, t_k\}$ be a transversal of $\mathrm{St}_G(v)$, so that $\rist'_G(n)=\prod_{i=1}^{k}\rist'_G(t_iv)$. Then, the subgroup generated by $\bigcup_{i=1}^{k}t_iHt_i^{-1}$ is a finitely generated subgroup of $N$, since $H$ is finitely generated and $N$ is normal, and it contains $\rist'_G(n)$, since it contains $t_i\rist'_G(v)t_i^{-1}=\rist'_G(t_iv)$ for all $1\leq i \leq k$.
	
	Let us consider the quotient $G/\rist'_G(n)$. By assumption, every subgroup of this group is finitely generated. In particular, $N/\rist'_G(n)$ is finitely generated. Since $\rist'_G(n)$ is contained in a finitely generated subgroup of $N$, we conclude that $N$ itself must also be finitely generated.
\end{proof}

\begin{cor} \label{cor-norm-branch}
	Every normal subgroup of a finitely generated branch group is finitely generated.
\end{cor}

\begin{proof}
	It suffices to check that a finitely generated branch group $G$ satisfies all the hypotheses of Theorem \ref{thm:NormalSubgroupsAreFG}. Such a group is obviously weakly branch. Furthermore, for all $n\in \N$, $\rist_G(n)$ is of finite index in $G$ by definition, and thus is finitely generated. The only thing left to verify is that $G/\rist'_G(n)$ is Noetherian for all $n\in \N$. However, since $\rist_G(n)$ is of finite index in $G$, the group $G/\rist'_G(n)$ is a finitely generated virtually abelian group, and it is well-known that such a group is Noetherian.
\end{proof}

{\small 
\bibliographystyle{abbrv}
\bibliography{MicroComm}
}
\end{document}